\documentclass{amsart}
\usepackage{mathtools}
\usepackage{amsmath,amssymb,amsthm,color,mathrsfs,array}
\usepackage{multirow}
\usepackage{verbatim}
\usepackage{url}
\usepackage{xcolor}
\usepackage{float}
\renewcommand{\arraystretch}{1.2}

\newcommand{\tallopenbox}{\leavevmode
  \hbox to.6em{%
  \hfil\vrule
  \vbox to.65em{\hrule width.35em\vfil\hrule}%
  \vrule\hfil}}

\newcommand{\re}{\mathbb{R}}
\newcommand{\mR}{\mathbb{R}}

\newcommand{\mC}{\mathbb{C}}

\newcommand{\cpx}{\mathbb{C}}

\newcommand{\N}{\mathbb{N}}

\newcommand{\diag}{\mbox{diag}}

\newcommand{\lmd}{\lambda}

\newcommand{\eps}{\epsilon}

\newcommand{\Dt}{\Delta}
\def\af{\alpha}

\def\rank{\mbox{rank}}

\newcommand{\reff}[1]{(\ref{#1})}

\newcommand{\mc}[1]{\mathcal{#1}}

\newcommand{\qmod}[1]{\mbox{QM}[#1]}
\newcommand{\ideal}[1]{\mbox{Ideal}[#1]}

\newcommand{\st}{\mathit{s.t.}}

\newcommand{\tr}[1]{\operatorname{tr}(#1)}

\newcommand{\uvec}[1]{\operatorname{uvec}(#1)}
\newcommand{\vect}[1]{\operatorname{vec}(#1)}

\usepackage{enumitem}
\newcommand{\bdes}{\begin{description}}
	\newcommand{\edes}{\end{description}}
\newcommand{\bal}{\begin{align}}
	\newcommand{\eal}{\end{align}}
\newcommand{\bnum}{\begin{enumerate}}
	\newcommand{\enum}{\end{enumerate}}
\newcommand{\bit}{\begin{itemize}}
	\newcommand{\eit}{\end{itemize}}
\newcommand{\bea}{\begin{eqnarray}}
	\newcommand{\eea}{\end{eqnarray}}
\newcommand{\be}{\begin{equation}}
	\newcommand{\ee}{\end{equation}}
\newcommand{\baray}{\begin{array}}
	\newcommand{\earay}{\end{array}}
\newcommand{\bsry}{\begin{subarray}}
	\newcommand{\esry}{\end{subarray}}
\newcommand{\bca}{\begin{cases}}
	\newcommand{\eca}{\end{cases}}
\newcommand{\bcen}{\begin{center}}
	\newcommand{\ecen}{\end{center}}
\newcommand{\bbm}{\begin{bmatrix}}
	\newcommand{\ebm}{\end{bmatrix}}
\newcommand{\bmx}{\begin{matrix}}
	\newcommand{\emx}{\end{matrix}}
\newcommand{\bpm}{\begin{pmatrix}}
	\newcommand{\epm}{\end{pmatrix}}
\newcommand{\btab}{\begin{tabular}}
	\newcommand{\etab}{\end{tabular}}

\theoremstyle{plain}
\newtheorem{theorem}{Theorem}[section]

\newtheorem{lem}[theorem]{Lemma}

\newtheorem{cor}[theorem]{Corollary}

\newtheorem{ass}[theorem]{Assumption}

\newtheorem*{claim*}{Claim}
\theoremstyle{definition}
\newtheorem{example}[theorem]{Example}

\numberwithin{equation}{section}
\numberwithin{table}{section}

\def\r{{\mathbb{R}}}

\def\n{{\mathbb{N}}}

\def\n{{\mathbb{N}}}

\begin{document}

\title{Lagrange multiplier  expressions for matrix polynomial optimization and tight relaxations
}

\author[Lei Huang]{Lei Huang}
\address{Lei Huang, Jiawang Nie, Jiajia Wang, Department of Mathematics,
University of California San Diego,
9500 Gilman Drive, La Jolla, CA, USA, 92093.}
\email{leh010@ucsd.edu,njw@math.ucsd.edu,jiw133@ucsd.edu}
\author[Jiawang Nie]{Jiawang~Nie}

\author[Jiajia Wang]{Jiajia~Wang}

\author[Lingling Xie]{Lingling Xie}

\address{Lingling Xie, School of Mathematics and Statistics, Hefei Normal University,
No. 327 Jinzhai Road, Hefei, Anhui, China, 230061.}

\email{linglingxie@hfnu.edu.cn}




\begin{abstract}
This paper studies matrix  polynomial optimization.
We investigate explicit expressions for Lagrange multiplier matrices from
the first order optimality conditions.  The existence of these expressions
can be shown under the nondegeneracy condition. Using  multiplier matrix expressions, we propose a strengthened Moment-SOS hierarchy for solving
matrix polynomial optimization. Under some general assumptions,
we show that this strengthened hierarchy is tight, or equivalently, it has finite convergence.
We also study how to detect tightness and  extract optimizers.
Preliminary numerical experiments demonstrate the efficiency of the strengthened hierarchy.
\end{abstract}

\maketitle

\section{Introduction}
Consider the matrix polynomial optimization (MPO)
\be  \label{nsdp}
\left\{ \baray{rl}
\min\limits_{ x\in \re^n } & f(x)  \\
\st &  	G(x) \succeq 0, \\
\earay \right.
\ee
where $f(x)$ is a polynomial in $x:=(x_1,\dots,x_n)$
and $G(x)$ is an $m$-by-$m$ symmetric matrix polynomial in $x$.
The inequality $G(x) \succeq 0$ means that the matrix $G(x)$ is positive semidefinite.
Denote by $K$ the feasible set of \reff{nsdp}.
Let $f_{\min}$ be the optimal value of \reff{nsdp}.
The MPO \reff{nsdp} has broad applications in various fields
\cite{HHP15,hdlb,hdlb2,HDLJ,klms,schhol,yzgf}.

When the matrix $G(x)$ is diagonal, \reff{nsdp} reduces to
classical scalar polynomial optimization with scalar inequality constraints.
The Moment-SOS hierarchy, proposed by Lasserre~\cite{Las01},
is a powerful approach for solving scalar polynomial optimization.
It consists of a sequence of semidefinite programming relaxations
using moments and sum-of-squares. The convergence theory of the Moment-SOS hierarchy
has been  extensively studied.  For its asymptotic convergence rates,
we refer to the recent survey \cite{lssurvey}. Under certain conditions, this hierarchy
has finite convergence; see the works
\cite{dKlLau11,hl23,hny2,Las09,LLR08,Lau07,nieopcd}.
We refer to the books and surveys \cite{HDLJ,LasBk15,Lau09,niebook,Sch09}
for more general introductions to polynomial optimization.

To solve the MPO \reff{nsdp}, a matrix Moment-SOS hierarchy was proposed in \cite{hdlb,schhol}.
When the quadratic module generated by $G(x)$ is Archimedean (a condition slightly stronger than compactness),
the matrix Moment-SOS hierarchy can be shown to have asymptotic convergence and the convergence rate  is studied in \cite{hl24}.	
Recently, it was shown in \cite{hn24} that this hierarchy also has finite convergence
if, in addition, the nondegeneracy condition, strict complementarity
condition and second order sufficient condition hold at every global minimizer of \reff{nsdp}.

The convergence rate of the classical Moment-SOS hierarchy can be slow,
or it may not have finite convergence for some problems \cite{hlm25,nieopcd,Tight18}.
In particular, when the feasible set is unbounded,
its asymptotic convergence cannot be guaranteed.
Optimality conditions are quite useful for strengthening Moment-SOS relaxations.
For unconstrained optimization, a strengthened Moment-SOS hierarchy
based on gradient ideal is proposed in \cite{NDS06}
and it has finite convergence.
Based on the Karush-Kuhn-Tucker (KKT) conditions,
similar strengthened Moment-SOS hierarchies are proposed in \cite{DNP07,niejac}
to solve constrained optimization problems.
However, these methods are not computationally convenient,
since they either introduce new variables for Lagrange multipliers or
use maximal minors of Jacobian matrices.
To overcome this issue, Lagrange multiplier expressions (LMEs) are introduced in \cite{Tight18},
and a strengthened Moment-SOS hierarchy is proposed by using those expressions.
This strengthened hierarchy has finite convergence under some genericity assumptions.
Recently, KKT type conditions have also been   introduced to obtain strengthened relaxations
for solving noncommutative polynomial optimization \cite{akgv23},
which rapidly improve the convergence speed.

There has been relatively little work on strengthening the matrix Moment-SOS hierarchy
for solving the MPO \reff{nsdp}
by incorporating  KKT conditions. A partial reason is that the constraint qualifications
and KKT conditions for \reff{nsdp} are much more complicated than the scalar case.
One might consider an alternative equivalent scalar formulation of \reff{nsdp}.
For instance, the matrix inequality $G(x)\succeq 0$ is equivalent to that
all the principal minors are nonnegative, and one might apply the
corresponding strengthened Moment-SOS hierarchies in \cite{DNP07,niejac,Tight18}.
However, this approach is not computationally efficient, as the degrees of the scalar describing polynomials are high \cite{hdlb}.
Even worse, scalarization may cause singularities
(e.g., the constraint qualifications may fail),
which could prevent    the construction of the corresponding strengthened relaxations or the resulting  hierarchy may fail to have finite convergence.
We refer to Section~\ref{sca:mat} for more detailed discussions about this.

The above observations motivate us to construct strengthened Moment-SOS relaxations
for solving \reff{nsdp} by directly incorporating its KKT conditions.
In particular, expressions for Lagrange multiplier matrices
are highly desirable  for solving matrix optimization problems.

\subsection*{Contributions}
To present our main results, we first introduce the KKT conditions for \reff{nsdp}.
Let $\mathcal{S}^{m}$ denote the space of all $m$-by-$m$ real symmetric matrices.
For the matrix function $G(x)= ( G_{st}(x))_{s,t=1,\ldots,m}$ as in \reff{nsdp},
its gradient at $x$ is the linear mapping
$\nabla G(x): \mathbb{R}^n  \rightarrow \mathcal{S}^m$
such that
\[
v   \coloneqq   \left(v_1, \ldots, v_n\right) \, \mapsto \,
\nabla G(x)[v] \, \coloneqq \, \sum_{i=1}^n v_i \nabla_{x_i} G(x),
\]
where $\nabla_{x_i} G(x) $ denotes the partial derivative  of $G$
with respect to $x_i$, i.e.,
\[
\nabla_{x_i} G(x) \, = \, \left(\frac{\partial G_{st}}{\partial x_i}\right)_{s,t=1,\dots,m}.
\]
The adjoint of $\nabla G(x)$ is the linear mapping
$\nabla G(x)^*: \mathcal{S}^m \rightarrow \mathbb{R}^n$ such that
\[
X \, \mapsto \, \nabla G(x)^*[X] \, \coloneqq \,
\bbm \tr{\nabla_{x_1} G(x) X }
& \cdots  & \tr{ \nabla_{x_n} G(x)  X } \ebm^T.
\]
Here, $\tr{A}$ stands for the trace of the matrix $A$.

Suppose $u$ is a minimizer of \reff{nsdp}.
Under the nondegeneracy condition (NDC) (see Section \ref{sec:cq} for details),
there exists a symmetric matrix
$\Lambda \in \mathcal{S}^m$ satisfying
\begin{equation}\label{kkt}
\left\{
\begin{array}{rcl}
\nabla f(u)-\nabla G(u)^*[\Lambda] &=& 0, \\
\Lambda \succeq 0, \, G(u) &\succeq& 0, \\
G(u) \Lambda   &=& 0.
\end{array}  \right.
\end{equation}
These are referred to as the KKT conditions or the first order optimality conditions (FOOC).
A point $u\in \mR^n$ satisfying \reff{kkt} is called a KKT point,  the matrix  $\Lambda$ is called the associated Lagrange multiplier matrix, and $(u,\Lambda)$ is called a KKT pair.
When \reff{kkt} holds, the problem \reff{nsdp} is equivalent to
\begin{equation} \label{int:nsdp:eq}
\left\{ \baray{cl}
\min\limits_{x\in \mR^n, \Lambda \in\mathcal{S}^{m}} & f(x)  \\
\st &   \nabla f(x)-\nabla G(x)^*[\Lambda] = 0, \\
&  G(x)\Lambda =0,~G(x)\succeq 0,~\Lambda \succeq 0.
\earay \right.
\end{equation}
The reformulation \reff{int:nsdp:eq} introduces a new matrix variable, the Lagrange multiplier matrix $\Lambda$, which has $\binom{m+1}{2}$ entries. The usage of $\Lambda$ in \reff{int:nsdp:eq}
 substantially increases the computational cost of the resulting Moment-SOS hierarchy.

In this paper, we express the Lagrange multiplier matrix $\Lambda$
as a symmetric matrix polynomial $\Theta(x)$.
The problem \reff{nsdp} is then equivalent to
\begin{equation} \label{int:nsdp}
\left\{ \baray{rl}
\min & f(x)  \\
\st &   \nabla f(x)-\nabla G(x)^*[\Theta(x)] = 0, \\
&  G(x)\Theta(x)=0, \, G(x)\succeq 0,\, \Theta(x)\succeq 0.
\earay \right.
\end{equation}
This reformulation has the same variable $x$ as in \reff{nsdp},
by using the expression $\Lambda = \Theta(x)$.
The problem \reff{int:nsdp} can be viewed as a matrix analogue of
the strengthened reformulation for polynomial optimization in \cite{Tight18}.
We refer to Section~\ref{sca:mat} for the major differences
between scalar constraints and matrix constraints.

By incorporating the KKT conditions and the expressions
of the Lagrange multiplier matrices,
  we propose a strengthened matrix Moment-SOS hierarchy
for solving \reff{nsdp}.
Interestingly, this strengthened hierarchy has finite convergence
under some general assumptions. Our major contributions are:
\bit
\item [i)] When the constraint matrix $G(x)$ is
nondegenerate (see \reff{G(x):nondgn}),
we show that the Lagrange multiplier matrix $\Lambda$
can be expressed as a matrix  polynomial $\Theta(x)$.
We further provide a numerical algorithm for computing $\Theta(x)$ efficiently.

\item [ii)]
The matrix Moment-SOS hierarchy is applied to solve the reformulation \reff{int:nsdp}, and we refer to this as the strengthened hierarchy. Under some general assumptions,
we prove that the strengthened hierarchy is tight, i.e., it has finite convergence.
In particular, we do not require the feasible set to be compact, nor assume that the strict complementarity condition or the second order optimality condition holds at every global minimizer.

\item [iii)] We show that every minimizer of the moment relaxation  has a flat truncation
when the relaxation order is sufficiently large, under some mild conditions.

\item [iv)]
We compare the performance of the strengthened matrix Moment-SOS hierarchy
based on \reff{int:nsdp} with that of the classical matrix Moment-SOS hierarchy based on \reff{nsdp}.
The numerical experiments demonstrate that the strengthened hierarchy  converges significantly faster.
\eit

This paper is organized as follows. Section~\ref{sc:pre} reviews some basics
of nonlinear semidefinite optimization and polynomial optimization.
Section~\ref{sec:exp} studies polynomial expressions for Lagrange multiplier matrices.
Section~\ref{sec:exs} presents the strengthened matrix Moment-SOS hierarchy
and summarizes the tightness results.
Section~\ref{sc:fct} gives the proof for the tightness of the strengthened hierarchy.
Section \ref{sc:detect} discusses how to detect tightness and extract minimizers.
Section~\ref{sec:num} presents some  numerical examples.
We draw conclusions in Section~\ref{sc:dis}.

\section{Preliminaries}
\label{sc:pre}

\noindent {\bf Notation}
The symbol $\mathbb{N}$ (resp., $\mathbb{R}$, $\cpx$) denotes the set of
nonnegative integers (resp., real numbers, complex numbers).
Let $\mathbb{R}[x]:=\mathbb{R}\left[x_1, \ldots, x_n\right]$ (resp., $\mathbb{C}[x]$)
be the ring of polynomials in $x:=\left(x_1, \ldots, x_n\right)$
with real coefficients (resp., complex coefficients), and let $\mathbb{R}[x]_d$
denote the set of polynomials in $\mathbb{R}[x]$ with degrees at most $d$.
For a positive integer $m$,  denote by  $\mathcal{S}^{m}$
(resp., $\mathcal{S}_{+}^m$, $\mathcal{S}_{\mC}^{m}$)
the set of all $m$-by-$m$  real symmetric matrices
(resp., positive semidefinite matrices, complex symmetric matrices).
For $\af = (\af_1, \ldots, \af_n) \in \N^n$ and an integer $d > 0$,
denote
\[
\Dt(d)  \, \coloneqq \,  \frac{d(d+1)}{2},
\quad |\af| \coloneqq \af_1 + \cdots + \af_n,\quad
\mathbb{N}_d^n \, \coloneqq \,  \left\{\alpha \in \mathbb{N}^n \mid
|\af| \leq d\right\},
\]
and  denote the monomial vector
\[
[x]_{d}:=[\begin{array}{llllllll}
1 & x_1 & \cdots & x_n & x_1^2 & x_1 x_2 & \cdots & x_n^{d}
\end{array}]^T.
\]
Denote by $\operatorname{deg}(p)$ the total degree of the polynomial $p$.
For a  matrix polynomial $H$,  $\operatorname{deg}(H)$
is the maximum degree among all its entries.
The symbol $\lceil t\rceil$ denotes the smallest integer that
is greater than or equal to $t$ and the symbol $I_{m}$
stands for the $m$-by-$m$ identity matrix.
For a matrix $A$, $A^{T}$ denotes its transpose and
$\ker{A}$ denotes the kernel of $A$ (i.e., the null space).
The trace of $A$ is denoted as $\tr{A}$,
which equals the sum of  its diagonal entries.
The  vectorization of $A\in \mR^{m\times n}$, denoted by $\vect{A}$,
is an $mn$-dimensional column vector formed by stacking the columns of $A$ from left to right.
When $A$ is symmetric, the vectorization of its upper triangular entries
is denoted as $\uvec{A}$. For a smooth function $f(x)$,
its partial derivative with respect to $x_i$ is denoted as $f_{x_i}$.

\subsection{Optimality conditions for nonlinear semidefinite optimization}
\label{sec:cq}
This subsection reviews the nondegeneracy condition and the first order optimality conditions for nonlinear semidefinite optimization. We refer to \cite{shap,shuni,sunde} for more details.

Consider the optimization problem of the form \reff{nsdp},
where  $f(x)$ and each entry of $G(x)$ are general smooth functions.
Let $u$ be a feasible point of \reff{nsdp}.
The (Bouligand) tangent cone to  $\mathcal{S}_{+}^m$ at $G(u)$ is the set
\[
T_{\mathcal{S}_{+}^m}(G(u))  = \left\{N \in \mathcal{S}^m: E^T N E \succeq 0\right\},
\]
where $E$ is a matrix whose column vectors form a basis of $\ker\,G(u)$.
The lineality space at $u$ is the largest subspace contained in
$T_{\mathcal{S}_{+}^m}(G(u))$, which is
\[
\operatorname{lin}\left(T_{\mathcal{S}_{+}^m}(G(u))\right) \, =  \,
\left\{N \in \mathcal{S}^m: \, E^T N E=0\right\}.
\]
The {\it nondegeneracy condition} (NDC)
is said to hold at $u$ if
\begin{equation}\label{CQ}
\operatorname{Im} \nabla G(u) +
\operatorname{lin} \Big( T_{\mathcal{S}_{+}^m}(G(u)) \Big)
=	\mathcal{S}^{m} .
\end{equation}
In the above, $\operatorname{Im} \nabla G(u)$
denotes the image of the mapping $\nabla G(u)$.
The NDC can be viewed as an analogue of the classical
linear independence constraint qualification condition
in nonlinear programming, as it implies the uniqueness of the  Lagrange multiplier matrix.
When $u$ is a local minimizer of \reff{nsdp} and the NDC holds at $u$,
there exists a unique matrix $\Lambda\in \mc{S}^m$ satisfying
\begin{equation}\label{kkteq}
\begin{array}{rcl}
\nabla f(u) -\nabla G(u)^*[\Lambda] &=& 0, \\
\Lambda \succeq 0, \, G(u) \Lambda   &=& 0.
\end{array}
\end{equation}

In the following, we give an equivalent description for the NDC,
which is useful for studying the existence of
Lagrange multiplier matrix expressions in Section \ref{sec:exp}.
For convenience, denote the null space
\[
\ker_{\mc{S}^m} G(u) \coloneqq \{ Y \in \mc{S}^m:  G(u)Y = 0 \}.
\]
By taking orthogonal complements, the condition \reff{CQ} is equivalent to
\begin{equation} \label{Range:NDC}
\operatorname{Im} \nabla G(u)^{\perp}  \cap
\Big( \operatorname{lin} \big( T_{\mathcal{S}_{+}^m}(G(u)) \big) \Big)^\perp
=	\{ 0 \}.
\end{equation}
Note that $\operatorname{Im} \nabla G(u)^{\perp} = \ker \nabla G(u)^*$.
One can also see that
\[
\Big( \operatorname{lin} \big( T_{\mathcal{S}_{+}^m}(G(u)) \big) \Big)^\perp
= \ker_{\mc{S}^m} G(u).
\]
Therefore, we get the following lemma.

\begin{lem}  \label{lm:ndc:cap=0}
The NDC \reff{CQ} holds at $u$ if and only if
\begin{equation}\label{NDC:cap=0}
\ker \nabla G(u)^* \cap \ker_{ \mc{S}^m } G(u) = \{ 0 \}.
\end{equation}
\end{lem}

A similar notion of nondegeneracy to that in  \reff{NDC:cap=0}
is given in  \cite[Definition~3.3.1]{SDPbook}.
We remark that the condition \reff{NDC:cap=0}
remains well-defined even when $u$ is not a feasible point of \reff{nsdp}.

\subsection{Some basics in real algebraic geometry}
\label{sec:qm}
This subsection reviews some basics in polynomial optimization.
We refer to \cite{HDLJ,Lau09,niebook,Sch09} for more detailed introductions.

A subset $I  \subseteq \mathbb{R}[x]$ is called an ideal of $\re[x]$
if $I \cdot \mathbb{R}[x] \subseteq I$, $I+I \subseteq I$.
Its complex variety is the set of complex common zeros of polynomials in $I$, i.e.,
\begin{equation*}
V_{\cpx}(I)=\{x \in \cpx^n \mid p(x)=0 ~ \forall p\in I\}.
\end{equation*}
Its real variety is the set $V_{\re}(I) = V_{\mC}(I) \cap \re^n$.
The ideal generated by a polynomial matrix  $H\in \mR[x]^{m_1\times m_2}$
is defined to be the ideal generated by all its entries, i.e.,
\begin{equation*}
\ideal{H} \,= \, \sum\limits_{i=1}^{m_1}\sum\limits_{j=1}^{m_2}H_{ij}(x)\cdot \mR[x].
\end{equation*}
For a polynomial matrix  tuple $\mathcal{H}  \coloneqq (H_1,\dots, H_q)$, we define
\begin{equation*}
\ideal{\mathcal{H}} \,= \, \ideal{H_1}+\cdots+\ideal{H_q}.
\end{equation*}
For a degree $k$, we define the truncations:
\begin{eqnarray}
\ideal{H}_k & = & \sum\limits_{i=1}^{m_1}\sum\limits_{j=1}^{m_2}
H_{ij}(x)\cdot \mR[x]_{k-\deg(H_{ij})}, \\
\ideal{\mathcal{H}}_k & = & \ideal{H_1}_k+\cdots+\ideal{H_q}_k.
\end{eqnarray}

A polynomial $p$ is said to be a sum of squares (SOS) if
$
p=p_1^2+\dots+p_t^2$ for $p_1,\dots,p_t \in \mathbb{R}[x]$.
The set of all SOS polynomials in $x$ is denoted as $\Sigma[x]$.
For a  degree $k$, denote the truncation
\[
\Sigma[x]_{k} \, \coloneqq  \, \Sigma[x] \cap  \mathbb{R}[x]_{k} .
\]
We refer to \cite{HN08,niebook} for matrix SOS polynomials.
The quadratic module of a symmetric matrix polynomial $G \in \re[x]^{m\times m}$ is the set
\begin{equation}\nonumber
\qmod{G}  \coloneqq  \Big \{
\sigma+\sum_{t=1}^r v_t^TGv_t \mid
\sigma \in \Sigma[x],~v_t \in \mathbb{R}[x]^m,~r\in \N
\Big \}.
\end{equation}
For a symmetric matrix polynomial tuple $\mathcal{G}:=(G_1,\dots,G_{s})$, we define
\[
\qmod{\mathcal{G}} \,  \coloneqq  \,  \qmod{G_1}+ \cdots+ \qmod{G_s}.
\]
For a degree $k$, we similarly define the truncations:
\be
\qmod{G}_{k}  \coloneqq   \left\{\baray{l|l}
\sigma+\sum\limits_{t=1}^{r} v_t^T G v_t& \baray{l}  \sigma \in \Sigma[x],~
v_t \in \mathbb{R}[x]^m,~r\in \N,\\
\deg(\sigma)\leq k,~\deg(v_t^T G v_t)\leq k
\earay
\earay \right\},
\ee
\be
 \qmod{\mathcal{G}}_{k}   \coloneqq   \qmod{G_1}_{k} +\cdots+ \qmod{G_s}_{k}.
\ee

The set $\ideal{\mathcal{H}}+\qmod{\mathcal{G}}$ is said to be Archimedean,
if  there exists $R >0$ such that $R-\|x\|^2 \in \ideal{\mathcal{H}}+\qmod{\mathcal{G}}$.
Clearly, for $p \in \ideal{\mathcal{H}}+\qmod{\mathcal{G}}$, we have $p(u) \geq  0$
for all $u \in  T = \{x\in \mR^n: H(x)=0,~\mathcal{G}(x)\succeq 0\}$.
Conversely, if $p$ is strictly positive on $T$ and $\ideal{\mathcal{H}}+\qmod{\mathcal{G}}$ is Archimedean,
then it holds $p \in \ideal{\mathcal{H}}+\qmod{\mathcal{G}}$.
This is referred to as the matrix version of Putinar's Positivstellensatz \cite{schhol}.

A truncated multi-sequence  $y \in \re^{ \mathbb{N}_{2k}^n }$ induces the Riesz functional
$\langle \cdot, y\rangle$ acting on $\mathbb{R}[x]_{ d}$ as
\begin{equation} \label{Reiz:fun}
\langle \sum_{\alpha \in \mathbb{N}_{2k}^n} p_{\alpha} x^{\alpha}, y\rangle
\,  \coloneqq  \, \sum_{\alpha \in \mathbb{N}_{2k}^n} p_{\alpha} y_{\alpha}.
\end{equation}
For a polynomial $q\in \mR[x]_{2k}$,
the $k$th order {\em localizing matrix} of   $q$ for $y$
is the symmetric matrix $L_{q}^{(k)}[y]$ satisfying
\begin{equation}\label{locmat:gi}
L_{q}^{(k)}[y] \, = \, \langle q\cdot [x]_{t}[x]_{t}^T, y\rangle,
\end{equation}
where $t = k-\lceil \deg(q)/2 \rceil$,
and $\langle \cdot, y\rangle$ is applied entry-wise to the  matrix $q\cdot [x]_{t}[x]_{t}^T$.
When $q = 1$, the localizing matrix $L_{q}^{(k)}[y]$
becomes the $k$th order {\it moment matrix}
$
M_k[y]\,\coloneqq\, L_1^{(k)}[y].
$
For a  matrix polynomial $H = (H_{ij}) \in \mR[x]^{m_1 \times m_2}$,
the localizing  matrix of $H$ for
$y\in \mR^{ \N^n_{2k} }$ is the block matrix
\[
L_{H}^{(k)}[y] \, \coloneqq \,
( L_{ H_{ij} }^{(k)}[y] )_{1 \le i\leq m_1, 1\le j \le m_2 },
\]
where each $L_{ H_{ij} }^{(k)}[y]$ is defined as in \reff{locmat:gi}. 	
Note that $L_{H}^{(k)}[y]$ is symmetric when $H$ is symmetric.

\section{Lagrange multiplier matrix expressions}
\label{sec:exp}
The KKT equation for \reff{nsdp} reads as
\begin{equation}\label{kktset:x}
\left\{ \begin{array}{rcl}
\nabla f(x) - \nabla G(x)^*[\Lambda] &=& 0, \\
G(x) \Lambda    &=& 0,
\end{array} \right.
\end{equation}
where $\Lambda= (\Lambda_{ij}) \in \mc{S}^m$ is the Lagrange multiplier matrix.
Recall that
\[
\nabla G(x)^*[\Lambda]  =  \bbm \tr{ \nabla_{x_1} G(x) \Lambda }
& \cdots  & \tr{ \nabla_{x_n} G(x) \Lambda } \ebm^T.
\]
Then, it holds that
\[
\nabla G(x)^*[\Lambda] = \sum_{i, j=1}^m  \Lambda_{i j} \nabla G_{i j}(x)
= \sum_{i=1}^m  \Lambda_{ii} \nabla G_{ii}(x)  +
2 \sum_{1 \le i<j \le m}  \Lambda_{ij} \nabla G_{ij}(x) .
\]
This is equivalent to
\[
\nabla G(x)^*[\Lambda]   =  P_1(x)\cdot \uvec{\Lambda},
\]
where the matrix polynomial $P_1(x) \in \mathbb{R}[x]^{n\times  \Dt(m)}$
is
\begin{equation*}
P_1 = \bbm
\nabla G_{11}  & 2\nabla G_{12} &
\nabla G_{22} &
2\nabla G_{13} & \cdots & 2\nabla G_{1m}   & \cdots &2\nabla G_{m-1,m}  &    \nabla G_{mm}
\ebm,
\end{equation*}
$\Dt(m)  =  \frac{m(m+1)}{2}$, and $\uvec{\Lambda}$ is the vectorization of upper triangular entries of $\Lambda$ {in column-stacked order}, i.e.,
\[
\uvec{ \Lambda } = \bbm
\Lambda_{11} & \Lambda_{12}& \Lambda_{22} &  \Lambda_{13} & \cdots & \Lambda_{1m}   & \cdots & \Lambda_{m-1, m} & \Lambda_{mm}
\ebm^{T } .
\]
The matrix equation $G(x)\Lambda=0$ is equivalent to
$\vect{G(x)\Lambda} = 0$.
Let $P_2(x)\in \mR[x]^{m^2\times \Dt(m)}$ be the matrix polynomial such that
\[
\vect{G(x)\Lambda} = P_2(x) \cdot \uvec{\Lambda} .
\]

Denote the matrix polynomial
\begin{equation}\label{def:PP}
P(x)=\left[\begin{array}{l}
P_1(x) \\
P_2(x)
\end{array}\right] \in \mR[x]^{(n+m^2)\times \Dt(m)} .
\end{equation}
{Then,} we get the following lemma.

\begin{lem}
The KKT equation \reff{kktset:x} is equivalent to
\begin{equation}  \label{def:P}
P(x)\cdot \uvec{\Lambda}=\left[\begin{array}{c}
\nabla f(x) \\
0 \\
\end{array}\right],
\end{equation}
where $0$ denotes the zero vector of length $m^2$.
\end{lem}
\vspace{.2cm}

When $P(x)$  has full column rank, we can get the rational expression
\[
\uvec{\Lambda}= \big( P(x)^TP(x) \big)^{-1}P(x)\left[\begin{array}{c}
\nabla f(x) \\
0 \\
\end{array}\right].
\]
 However, this rational expression is often expensive to use in practice, since its denominator is typically a high-degree polynomial. Under the nonsingularity of $P(x)$, we can  obtain a polynomial expression.
The matrix polynomial $P(x)$ is said to be {\it nonsingular} if $\rank\, P(x) = \Dt(m)$
for every complex vector $x\in \mathbb{C}^n$;
that is, $P(x)$ has full column rank everywhere in $\mathbb{C}^n.$
When $P(x)$ is nonsingular, it follows from \cite[Proposition 5.2]{Tight18}
that there exists a matrix polynomial
$L(x)\in \mR[x]^{\Dt(m)\times (n+m^2)}$ such that
\begin{equation}\label{l:iden}
L(x)P(x)   =   I_{\Dt(m)}.
\end{equation}
The above  implies
\begin{equation}\label{eqlme}
\uvec{ \Lambda }= L(x)P(x)\cdot \uvec{ \Lambda } = L(x)\left[\begin{array}{c}
\nabla f(x) \\
0 \\
\end{array}\right].
\end{equation}
Let $\Theta(x)$ be the $m$-by-$m$ symmetric matrix polynomial such that
\begin{equation} \label{Theta(x)}
\uvec{ \Theta(x) }  =  L(x) \bbm \nabla f(x) \\ 0 \ebm .
\end{equation}
Therefore, for every pair $(x, \Lambda)$ satisfying \reff{kktset:x},
we have $\Lambda = \Theta(x)$. The matrix $\Theta(x)$
is called a {\it Lagrange multiplier matrix expression}.

The nonsingularity of $P(x)$ is closely related to the nondegeneracy condition \reff{CQ}.
By Lemma~\ref{lm:ndc:cap=0}, the NDC holds at a feasible point $u$ of \reff{nsdp}
if and only if the condition \reff{NDC:cap=0} holds at $u$. Note that \reff{NDC:cap=0}
remains well-defined even when  $u$ is not a feasible point of \reff{nsdp},
or when $u \in \cpx^n$ is a complex vector.
Therefore, we say the matrix polynomial $G(x)$ is {\it nondegenerate} if
\begin{equation}\label{G(x):nondgn}
\ker \nabla G(x)^* \,  \cap \, \ker_{ \mc{S}_{\cpx}^m } G(x) \ = \   \{ 0 \}
\quad \forall \, \, x \in \cpx^n.
\end{equation}
Here,
\[
\ker_{ \mc{S}_{\cpx} ^m } G(x)  = \{ Y \in \mc{S}_{\cpx}^m: G(x) Y = 0 \} .
\]
\begin{theorem}  \label{nonsing}
Let $P(x)$ be the matrix polynomial as in \reff{def:PP}.
Then $G(x)$ is nondegenerate if and only if $P(x)$ is nonsingular, i.e.,
$P(x)$ has full column rank for every complex vector $x \in \cpx^n$.
\end{theorem}
\begin{proof}
By the construction of $P(x)$ , one can see that
\[
\nabla G(x)^*[\Lambda] = 0, \quad  G(x) \Lambda  = 0,
\]
if and only if $P(x) \cdot \uvec{\Lambda} = 0$.
Therefore, the nondegeneracy condition \reff{G(x):nondgn}
holds if and only if $P(x)$ is nonsingular.
\end{proof}

In \reff{Theta(x)}, the matrix polynomial $\Theta(x)$ can be obtained by computing
$L(x)$ as in \reff{l:iden}.  We remark that  $L(x)$
can be determined by solving a linear system.
Let $v_i(x)^T$ denote the $i$th row of $L(x)$.
Then $L(x)P(x)=I_{\Delta(m)}$ is equivalent to
\begin{equation}\label{numeq}
v_i(x)^TP(x)=e_i^T,~~i=1,\dots,\Delta(m),
\end{equation}
where $e_i$ is the $i$th canonical basis vector in $\mR^{\Delta(m)}$.
Let $v^{(i)}$ denote the vector of concatenating coefficient vectors of all entries of $v_i(x)$.
For a priori degree $\ell$ of $L(x)$, \reff{numeq} is essentially
a linear system in the entries of $v^{(i)}$.
When $P(x)$ is nonsingular, the linear system \reff{numeq}
is solvable for sufficiently large  $\ell$,  as shown by the above analysis.
When it is solvable, \reff{numeq} may not have a unique solution.
In computational practice, we consider its minimum $2$-norm solution.
This is equivalent to the optimization problem ($i=1,\dots,\Delta(m)$)
\begin{equation}
\label{VL}
\left\{
\begin{array}{cl}
\min &  \|v^{(i)}\|_{2} \\
\st  &  v_i(x)^T P(x)  =  e_i^T  ,
\end{array}
\right.
\end{equation}
where $\|v^{(i)}\|_{2}$ denotes the $2$-norm of $v^{(i)}$.
The optimal solution of \reff{VL} can be explicitly given by the pseudoinverse of the coefficient matrix,
since it is a linear least square problem \cite{Demmel}.
Typically, we choose a priori degree $\ell$.
If \reff{VL} is infeasible for this $\ell$, we increase $\ell$ and solve it again.
We remark that, since our goal is to find a feasible $L(x)$ for \reff{numeq},
the choice of norm in \reff{VL} is not crucial, and in principle any norm would suffice.
We use the $\ell^2$-norm primarily for convenience, as \reff{VL} then becomes a linear least squares problem.
While using the $\ell^1$-norm may potentially help get sparse LMEs, it turns \reff{VL} into a linear program, which is generally more computationally expensive than linear least squares. Exploring sparse LMEs for sparse matrix polynomial optimization would be an interesting direction for future work.

For many cases, expressions for $L(x)$ and $\Theta(x)$
can be given explicitly.

\begin{example}\label{exmLME}
(i) Consider the linear matrix inequality
\[
G(x)=	\left[\begin{array}{cccc}
a_{11}^Tx+b_{11}&  a_{12}^Tx+b_{12}&\cdots&a_{1m}^Tx+b_{1m}  \\
a_{12}^Tx+b_{12}&a_{22}^Tx+b_{22} &\cdots& a_{2m}^Tx+b_{2m}  \\
\vdots&\vdots&\ddots &\vdots\\
a_{1m}^Tx+b_{1m}& a_{2m}^Tx+b_{2m}&\cdots& a_{mm}^Tx+b_{mm}  \\
\end{array}\right]\succeq 0,
\]
where the vectors $a_{ij} \in \mR^n$ $(1\leq i\leq j\leq m)$
are linearly independent and $b_{ij}$ are scalars.
The matrix $P_1$ reads as
\[
P_1=\begin{bmatrix}
a_{11} & 2a_{12} & a_{22} & 2a_{13} & \cdots & 2a_{1m}  & \cdots & 2a_{m-1,m}&a_{mm}
\end{bmatrix},
\]
which  is a real constant matrix of full column rank.
Hence, $P_1^TP_1$ is invertible and  the equation $L(x)P(x)=I_{\Dt(m)}$ is satisfied for
\[
L(x)=\bbm (P_1^TP_1)^{-1}P_1^T & \quad 0 \ebm.
\]
Then, we  obtain the Lagrange multiplier matrix expression
$\Theta(x)$, which satisfies
\[
\uvec{\Theta(x)} \, =  \, (P_1^TP_1)^{-1}P_1^T  \nabla f(x).
\]

\item[(ii)]	Consider the polynomial matrix inequality
\[
G(x)=\begin{bmatrix}
x_1^2-2 & \frac{1}{2}x_1x_2 \\
\frac{1}{2}x_1x_2 & x_2^2-2 \\
\end{bmatrix}\succeq 0.
\]
We have
\[
P_1(x)=\begin{bmatrix}
2x_1 & 0\\
x_2 & x_1\\
0 & 2x_2
\end{bmatrix}^T,\quad
P_2(x)=\begin{bmatrix}
x_1^2-2 & \frac{1}{2}x_1x_2 & 0 & 0\\
\frac{1}{2}x_1x_2 & x_2^2-2 & x_1^2-2 & \frac{1}{2}x_1x_2\\
0 & 0& \frac{1}{2}x_1x_2 & x_2^2-2
\end{bmatrix}^T.
\]
The equation $L(x)P(x)=I_{3}$ is satisfied for
\[
L(x)=\begin{bmatrix}
\frac{1}{4}x_1 & 0 & -\frac{1}{2}
& 0 & 0 & 0\\
\frac{1}{4}x_2-\frac{3}{32}x_1^2x_2 & \frac{1}{4}x_1-\frac{3}{32}x_1x_2^2 & \frac{3}{16} x_1x_2 & -\frac{1}{4} & -\frac{1}{4} & \frac{3}{16}x_1x_2 \\
0 & \frac{1}{4}x_2 & 0 & 0&0&-\frac{1}{2}\end{bmatrix}.
\]
Thus, we get the Lagrange multiplier matrix expression of $\Theta(x)$ as
{
\small
\[
\begin{bmatrix}
\frac{1}{4}x_1 f_{x_1} & (\frac{1}{4}x_2-\frac{3}{32}x_1^2x_2)f_{x_1}+(\frac{1}{4}x_1-\frac{3}{32}x_1x_2^2)f_{x_2}\\
(\frac{1}{4}x_2-\frac{3}{32}x_1^2x_2)f_{x_1}+(\frac{1}{4}x_1-\frac{3}{32}x_1x_2^2)f_{x_2} & \frac{1}{4}x_2f_{x_2}
\end{bmatrix}.
\]
}

\item[(iii)]  Consider the polynomial matrix inequality
\[
G(x)=\left[\begin{array}{ll}
x_1 & x_1x_2-1\\
x_1x_2-1& x_2x_3-1\\
\end{array}\right]\succeq 0.
\]
We have
\[
P_1(x)=\begin{bmatrix}
1 & 0&0\\
2x_2 & 2x_1&0\\
0 & x_3 &x_2\\
\end{bmatrix}^T,\,
P_2(x)=\begin{bmatrix}
x_1 & x_1x_2-1 & 0 & 0\\
x_1x_2-1 & x_2x_3-1 &  x_1 &x_1x_2-1\\
0 & 0 & x_1x_2-1 & x_2x_3-1
\end{bmatrix}^T.
\]
The equation $L(x)P(x)=I_{3}$ is satisfied for
\[
L(x)=\begin{bmatrix}
1-x_1x_2 & 0 & 1-x_2x_3&x_2&0&0&x_2 \\
x_1 & -\frac{1}{2}x_2 & \frac{1}{2}x_3 &-1&0&0&0 \\
-x_1 & x_2 & 0& 1&0& 0 &-1 \\
\end{bmatrix}.
\]
Then, we obtain the Lagrange multiplier matrix expression
\[
\Theta(x)=	
\begin{bmatrix}
(1-x_1x_2)f_{x_1}+(1-x_2x_3)f_{x_3}	&  x_1f_{x_1} -\frac{1}{2}x_2f_{x_2}-\frac{1}{2}x_3f_{x_3}\\
x_1f_{x_1} -\frac{1}{2}x_2f_{x_2}-\frac{1}{2}x_3f_{x_3}	& x_2f_{x_2}-x_1f_{x_1} \\
\end{bmatrix}.
\]
\end{example}	

\hspace{.1cm}

\begin{example}
\label{exmrandom}
We explore the minimum degree $\ell$ of the matrix polynomial $L(x)$
satisfying \reff{l:iden}. 
Consider the matrix polynomial
\[
G(x) =
\sum\limits_{\alpha\in \mathbb{N}^n_d} G_{\alpha}x^{\alpha},
\]
where the coefficient matrix $G_{\alpha}\in\mathcal{S}^{m}$
is generated randomly (its upper triangular entries are randomly generated
by \texttt{randn} in MATLAB).
As shown in Example~\ref{exmLME} (i), when $n\ge \Delta(m)$ and $d=1$,
the matrix $L(x)$ can be chosen to be constant, i.e., it has degree zero.
For general cases, it is mostly an open question for determining the minimal degree  of  $L(x)$.
Below, we explore the minimum degree $\ell$ by numerically solving \reff{VL},
for several typical values of $(m,n,d)$.
In Table~\ref{table:degG}, we list all minimum degrees $\ell$ of $L(x)$
satisfying \reff{l:iden} for $20$ randomly generated instances.

\begin{table}[htbp]
\centering
\caption{The minimum degree $\ell$ of $L(x)$ such that $L(x)P(x)=I_{\Delta(m)}$}
\label{table:degG}
\renewcommand{\arraystretch}{1.1}
\begin{tabular}{|c|c|c|c|c|c|} \hline
$(m,n,d)$   &    (3,2,1)   & (3,3,1)  &  (3,4,1)   & (3,5,1)  &  (4,5,1) \\ \hline
$\ell$   & 1, 2  &  1, 2 &  2, 3  &   1 &  3 \\ \hline
$(m,n,d)$ &   (4,6,1)&(4,8,1)& (4,9,1) & (5,12,1)& (5,14,1) \\  \hline
$\ell$ &3&2&1&2&1  \\  \hline
$(m,n,d)$ &  (6,20,1) & (2,3,2)  & (2,4,2)&(2,6,2)&(3,2,2) \\ \hline
$\ell$ & 1, 2 &3, 4&4&4, 5&5, 6, 7  \\  \hline
\end{tabular}
\end{table}
\end{example}

\section{The strengthened matrix Moment-SOS hierarchy}
\label{sec:exs}

In this section, we give a strengthened Moment-SOS hierarchy for solving \reff{nsdp}
by using the Lagrange multiplier matrix expression.
For this purpose, we make the following assumption.
\begin{ass}\label{ass:multi}
There exists a symmetric $m$-by-$m$ matrix polynomial $\Theta(x)$ such that
$\Lambda=\Theta(x)$ for every KKT pair $(x,\Lambda)$ of \reff{nsdp}.
\end{ass}

Assumption~\ref{ass:multi} is generally satisfied.
In particular, if the matrix polynomial $P(x)$ as in \reff{def:P} is nonsingular,
then Assumption \ref{ass:multi} holds (see Section~\ref{sec:exp}). We remark, however,
that the nonsingularity of $P(x)$ is a sufficient
but not necessary condition for  Assumption \ref{ass:multi} to be satisfied.
For instance, consider the case that $m=n=1$ and the optimization problem
\begin{equation} \nonumber
\left\{ \baray{rl}
\min\limits_{x\in \mR^1} & f(x)  \\
\st &  	(x^2+1)^2(x-1)\geq 0. \\
\earay \right.
\end{equation}
One can verify that $P(x)$  in \reff{def:PP} for the above problem is not nonsingular
since $P(\sqrt{-1}) = 0$.
Nevertheless, the nondegeneracy condition \reff{CQ} holds for every $x\in \mR$, and \reff{kktset:x} reduces to
\[
\nabla f(x)=\Lambda\cdot  (4x(x^2+1)(x-1)+(x^2+1)^2),\quad  \Lambda\cdot (x^2+1)^2(x-1)=0.
\]
If $x=1$, then $\Lambda=\frac{1}{4}\nabla f(1)$; if $x\neq 1$, the second equation forces $\Lambda=0$, and hence $\nabla f(x)=0$. Therefore,  $\Theta(x)=\frac{1}{4}\nabla f(x)$ is a  polynomial
multiplier expression
satisfying Assumption~\ref{ass:multi}.

When the optimal value $f_{\min}$ is achievable at a KKT point,
the problem \reff{nsdp} is equivalent to
\begin{equation} \label{sec3:nsdp:eq}
\left\{ \baray{cl}
\min\limits_{x\in \mR^n, \Lambda \in\mathcal{S}^{m}} & f(x)  \\
\st &   \nabla f(x)-\nabla G(x)^*[\Lambda] =  0, \\
&   G(x)\Lambda =0, \,	G(x)\succeq 0,~\Lambda \succeq 0 .
\earay \right.
\end{equation}
Under Assumption \ref{ass:multi}, the problem \reff{nsdp} is also equivalent to
\begin{equation} \label{nsdp:equ}
\left\{ \baray{rl}
\min & f(x)  \\
\st &    \nabla f(x)-\nabla G(x)^*[\Theta(x)] = 0, \\
& G(x)\Theta(x)=0, ~	G(x)\succeq 0, ~   \Theta(x)\succeq 0 .
\earay \right.
\end{equation}
We apply the matrix Moment-SOS hierarchy to solve \reff{nsdp:equ}. Denote the degree
\[
d_{0}  \, \coloneqq \,  \max{\left\{\lceil  \deg(f)/2 \rceil, \lceil
\deg(G)/2 +\deg(\Theta)/2 \rceil\right\}}.
\]
For a degree $k\geq  d_0$, the $k$th order moment relaxation  of \reff{nsdp:equ} is
\begin{equation} \label{mom}
\left\{ \baray{rl}
\min & \langle f, y\rangle  \\
\st &L^{(k)}_{G \Theta}(y)=0,~L^{(k)}_{\nabla f-\nabla G^*[\Theta] }(y)=0,\\
&  L_{G}^{(k)}(y) \succeq 0,~L_{\Theta}^{(k)}(y) \succeq 0, \\
&  y_0 =1, ~M_k(y) \succeq 0, \, y \in \re^{ \N_{2k}^n }.
\earay \right.
\end{equation}
The dual optimization problem of \reff{mom} is the $k$th order SOS relaxation:
\begin{equation} \label{sos}
\left\{ \baray{rl}
\max & \gamma \\
\st &f-\gamma \in \ideal{G \Theta, \nabla f-\nabla G^*[\Theta] }_{2k}+\qmod{G,\Theta}_{2k}.
\earay \right.
\end{equation}
We refer to Section~\ref{sec:qm} for the notation
and refer to \cite{hdlb,niebook} for the  duality.
When $k = d_0, d_0+1, \ldots$, the sequence of relaxations \reff{mom}--\reff{sos}
is called a strengthened matrix Moment-SOS hierarchy for solving \reff{nsdp}.

\subsection{Tightness of the strengthened matrix Moment-SOS hierarchy}
Let $f_{k,\mathrm{mom}}$ and $f_{k,\mathrm{sos}}$ denote the optimal values of
\reff{mom} and \reff{sos} respectively.
By the weak duality, we have
\[
f_{k,\mathrm{sos}} \le f_{k,\mathrm{mom}}  \le f_{\min}.
\]
The hierarchy of \reff{mom}--\reff{sos} is said to have asymptotic convergence if
$f_{k,\mathrm{sos}} \rightarrow f_{\min }$ as $k \rightarrow \infty$.
It is said to be {\it tight} (or to have finite convergence)
if there exists an order $N>0$ such that $f_{k,\mathrm{sos}}=f_{k,\mathrm{mom}}=f_{\min}$ for all $k \geq N$.
Our main result is that the hierarchy of \reff{mom}--\reff{sos}
is tight, under some general assumptions.
\begin{theorem}  \label{thm:finite}
Suppose the optimal value $f_{\min}$ of \reff{nsdp} is achievable at a KKT point,
Assumption \ref{ass:multi} holds, and
$\ideal{G\Theta,~\nabla f-\nabla G^* [\Theta]}+\qmod{G,\Theta}$ is Archimedean.
Then, the strengthened matrix Moment-SOS hierarchy of \reff{mom}--\reff{sos}
is tight, i.e., $f_{k,\mathrm{sos}}=f_{k,\mathrm{mom}}=f_{\min}$ for all $k$ sufficiently large.
\end{theorem}

The set $\ideal{G\Theta,~\nabla f-\nabla G^*[\Theta]}+\qmod{G,\Theta}$
is Archimedean if $\qmod{G}$ is Archimedean,
since the latter is a subset of the former. Moreover, this set is
Archimedean if the set of all pairs $(x,\Lambda)$ satisfying \reff{kktset:x} is compact.
In fact, for generic polynomial optimization problems, the set of KKT points is finite; this property holds on a Zariski open subset of the space of input polynomials with given degrees \cite{niebook}.
We remark that the assumption that $f_{\min}$ is achievable at a KKT point
cannot be dropped, as the optimal value of \reff{nsdp} may not be achieved at any KKT point, or may not even be achieved at any feasible point.
The proof of Theorem~\ref{thm:finite} is postponed to  Section~\ref{sc:fct}.

\subsection{A comparison with scalar inequality constraints}
\label{sca:mat}
The hierarchy of \reff{mom}--\reff{sos} can be viewed as the matrix analogue
of the strengthened Moment-SOS hierarchy in \cite{Tight18},
which is used to solve scalar polynomial optimization.
There are substantial differences in the algebraic and geometric properties of the two hierarchies.

When $G = \diag(g_1, \ldots,g_m)$ is diagonal, the MPO \reff{nsdp} is equivalent to
\begin{equation} \label{sp0}
\left\{ \baray{rl}
\min & f(x)  \\
\st &  	g_1(x)\geq 0,\cdots, g_m(x)  \geq 0. \\
\earay \right.
\end{equation}
The linear independence constraint qualification condition (LICQC)
holds at a minimizer $u$ of \reff{sp0} if $\rank \, C(u) = m$, for the matrix polynomial
\[
C(x)  \ = \
\bbm
\bmx \nabla g_1(x) &   \cdots & \nabla g_m(x) \emx   \\
\diag(g_1(x), \ldots, g_m(x))
\ebm \in \re[x]^{(m+n) \times m }.
\]
When the LICQC holds at $u$,  there exists
$\lambda = (\lambda_1,\dots,\lambda_m)\geq 0$ such that
\begin{equation}\label{scakkt}
\bbm
\bmx \nabla g_1(u) &   \cdots & \nabla g_m(u) \emx   \\
\diag(g_1(u), \ldots, g_m(u))
\ebm
\lmd  \ =  \  \bbm \nabla f(u) \\ 0 \ebm.
\end{equation}
The constraining tuple $g = (g_1, \ldots, g_m)$ is nonsingular
if $\rank \, C(x) = m$ for all $x\in \cpx^n$.
When it is nonsingular, there exist polynomials $p_1(x)$, \ldots,  $p_m(x)$ such that
$\lmd_i = p_i(u)$ for every $(u, \lmd)$ satisfying \reff{scakkt} \cite{Tight18}.
Then, \reff{sp0} is equivalent to
\begin{equation} \label{sp01}
\left\{ \baray{rl}
\min & \quad f(x)  \\
\st & \quad   p_1(x)\nabla g_1(x)+ \cdots  + p_m(x) \nabla g_m(x) - \nabla f(x) = 0,\\
&  \quad 	0 \le g_i(x)  \perp  p_i(x)  \geq 0,\, i = 1,\ldots, m.
\earay \right.
\end{equation}
In the above,  the symbol $\perp$ means that the product is zero.
It was shown in \cite{Tight18} that the Moment-SOS hierarchy for  solving  \reff{sp01}
is tight under  genericity assumptions.

The hierarchy \reff{mom}--\reff{sos} and Theorem \ref{thm:finite}	
can be viewed as the matrix analogue of the results  in \cite{Tight18}.
However, there are significant differences between their properties:
\bit
\item One may reformulate the matrix inequality $G(x)\succeq 0$ equivalently by scalar inequalities
(e.g., by requiring all the principal minors to be nonnegative),
and then apply the strengthened Moment-SOS hierarchy to solve scalar reformulation.
However, scalarization  often introduces singularities, i.e.,
the LICQC fails. For instance,  consider
\[
G(x)=\left[\begin{array}{ll}
x_1  & ~x_2\\
x_2   &~x_3 \\
\end{array}\right].
\]
The constraint $G(x)\succeq  0$ is equivalent to the following
scalar inequalities:
\[
x_1\geq 0,~x_3\geq  0,~x_1x_3-x_2^2\geq 0.
\]	
However, the LICQC fails at $0$ for these scalar constraints, while the NDC holds.
%
%
We typically cannot get Lagrange multiplier  expressions through scalarization.
\item
The proof for the tightness of the hierarchy \reff{mom}--\reff{sos}
(see Theorem \ref{thm:finite}) heavily relies on the matrix structure.
The proof techniques in \cite{Tight18} do not generalize directly
to the matrix case. A major reason is that the matrix equation
$G(x) \Lambda = 0$ does not imply $G(x) = 0$ or $\Lambda = 0$.
To prove the tightness of the hierarchy \reff{mom}--\reff{sos},
we need a more dedicated trick for $G(x) \Lambda = 0$.
Moreover, the proofs in \cite{Tight18} rely on Stengle's Positivstellensatz
for basic closed semialgebraic sets \cite{gsten},
which is not applicable for the matrix case.
\eit
We also refer to \cite[Section 3.2]{hn24} for more comparisons.

\section{Proof of tightness}
\label{sc:fct}
In this section, we give the proof for the tightness of
the strengthened Moment-SOS hierarchy of relaxations \reff{mom}--\reff{sos},
as stated in Theorem \ref{thm:finite}. For convenience,
 denote the set of complex critical pairs for  \reff{nsdp} by
\begin{equation}\label{kktvar}
\mathcal{K}^{\wedge}  \coloneqq  \left\{
\begin{array}{l|l}
(x,\Lambda) \in \mathbb{C}^n\times \mathcal{S}_{\mC}^{m} &
\begin{array}{r}
\nabla f(x)-\nabla G(x)^*[\Lambda ] = 0, \\
G(x)\Lambda=0
\end{array}
\end{array}\right\}.
\end{equation}
We complete the proof of Theorem \ref{thm:finite} in three major steps:

\begin{description}[leftmargin=2.2em, labelsep=0.5em, align=left]

\item[{\rm Step 1.}]
We show that $f$ achieves a constant value on each irreducible subvariety of $\mathcal{K}^{\wedge}$. This is shown in Lemma~\ref{kktfini}.

\item[{\rm Step 2.}] We express the set $\mathcal{K}^{\wedge}$
as a union of several subvarieties, each of which either contains no real points or on which $f$ attains a real constant value. For each subvariety, we prove an SOS-type representation for $f-f_{\min}$, as shown in Lemma~\ref{idealcons}.

\item[{\rm Step 3.}]
We aggregate the SOS-type representations for $f-f_{\min}$
on individual subvarieties to obtain a single one for $f-f_{\min}$.

\end{description}

\subsection{Step 1: some properties of the critical set}

\begin{lem}\label{kktfini}                                 	
Let $W$ be an  irreducible subvariety of $\mathcal{K}^{\wedge}$.
Then, $f$ achieves a constant value on $W$.
\end{lem}

\begin{proof}
Since the subvariety $W$ is irreducible, it is connected in $\mathbb{C}^{n}$,
and hence also path-connected (see \cite[4.1.3]{sawc}).
Define the Lagrangian function in
$(x,\Lambda) \in \mathbb{C}^n\times \mathcal{S}_{\mC}^{m}$:
\[
\mathcal{L}(x, \Lambda)=f(x)- \tr{ G(x)  \Lambda }.
\]
Note that  $\mathcal{L}(x, \Lambda)=f(x)$ for all $(x, \Lambda)\in \mathcal{K}^{\wedge}$.
For two arbitrary points $(x^{(1)}, \Lambda^{(1)})$, $(x^{(2)}, \Lambda^{(2)})$ in $W$,
we show that $f(x^{(1)})=f(x^{(2)})$.
\smallskip 		

Since $W$ is path-connected, there exists a piecewise-smooth curve
$(x(t),\Lambda(t))$ $(0 \leq t \leq 1)$ contained in $W$ such that
\[
(x(0),\Lambda(0))=(x^{(1)},\Lambda^{(1)}),~ (x(1),\Lambda(1))=(x^{(2)}, \Lambda^{(2)}),
\]
with a partition $0 =: t_0 < \cdots < t_N := 1$ such that $(x(t),\Lambda(t) )$
is smooth in each interval $(t_i, t_{i+1})$. For each interval $(t_i, t_{i+1})$,
we show the derivative of $f(x(t))$ at $\bar{t}$ is zero for all $\bar{t}\in (t_i, t_{i+1})$.
Define the sets
\[
S_m  \coloneqq \{t\in (t_i, t_{i+1})\mid \rank\, G(x(t)) =m\}, ~T_m  \coloneqq  cl(S_m)\backslash \{t_i, t_{i+1}\}.
\]
For $\ell=m-1, m-2, \ldots, 1, 0$, let
\[
S_\ell \coloneqq \{t\in (t_i, t_{i+1})\mid \rank\, G(x(t)) =\ell\}\backslash
\cup_{{j=\ell+1}}^m T_j,    ~T_\ell  \coloneqq cl(S_\ell)\backslash \{t_i, t_{i+1}\}.
\] 	

For $\bar{t}\in S_m$,  the matrix $G(x(\bar{t}))$ is invertible. There exists $\epsilon>0$
such that $G(x(t))$ remains invertible for all $t\in (\bar{t}-\epsilon,\bar{t}+\epsilon)$.  Since $G(x(t))\Lambda(t)=0$, it follows that $\Lambda(t)= 0$
for all $t\in (\bar{t}-\epsilon,\bar{t}+\epsilon)$, so
$\frac{\mathrm{d} \Lambda}{\mathrm{d}t}(\bar{t})= 0$ and
\[
\frac{\mathrm{d} f(x(\bar{t})) }{\mathrm{d}t} =
\frac{\mathrm{d} \mathcal{L}(x(\bar{t}))}{\mathrm{d}t} =
\Big( \nabla f(x(\bar{t}))- \nabla G(x(\bar{t}))^*[\Lambda(\bar{t})] \Big)^T \nabla x(\bar{t})=0.
\]
Since $\bar{t}$ is arbitrary in  $S_m$, we conclude that  $f(x(t))$ has  zero gradient
for all $t\in S_m$. The piecewise smoothness of  $x(t)$ implies that $f(x(t))$
also has  zero gradient for all $t\in T_m$.

For $\bar{t}\in S_{\ell}$ with $0 \le \ell \le m-1$, we have $\rank\, G(x(\bar{t}))=\ell$.
Up to a congruent permutation of rows and columns, we can write that
\[
G(x(\bar{t})) =  \left[\begin{array}{cc}
G_1(x(\bar{t}))     &   G_2(x(\bar{t}))   \\
G_2(x(\bar{t}))^T    &   G_3(x(\bar{t}))  \\
\end{array}\right],
\]
where $G_1(x(\bar{t}))$ is an invertible $\ell$-by-$\ell$
complex symmetric matrix. Let
\begin{eqnarray*}
Q(x(t)) & = &  \left[\begin{array}{cc}
I_{\ell}  &   -G_1(x(t))^{-1}G_2(x(t))   \\
0         &   I_{m-{\ell}}    \\
\end{array}\right], \\
S(x(t)) & =  &  G_3(x(t)) - G_2(x(t))^T  G_1(x(t))^{-1}  G_2(x(t)).
\end{eqnarray*}
Then, we have
\[
Q(x(t))^T  G(x(t)) Q(x(t))    =
\left[\begin{array}{cc}
G_1(x(t)) &   0         \\
0         &   S(x(t))   \\
\end{array}\right].
\]
Since $\bar{t}\notin \cup_{j={\ell}+1}^m T_j$,
there exists $\epsilon>0$ such that for all
$t \in (\bar{t}-\epsilon,\bar{t}+\epsilon)$,
\[
\rank\, G_1(x(t)) =  \ell,  \
\rank\, G(x(t))  \, \leq \, \rank\, G(x(\bar{t})) =  \ell  .
\]
Since $\rank\, G(x(t)) \ge \rank\, G_1(x(t))$, we have
\[
\rank\, G(x(t))=\rank\, G_1(x(t))={\ell} \quad \forall \,
t\in (\bar{t}-\epsilon,\bar{t}+\epsilon) .
\]
This implies that  $S(x(t))= 0$ for all $t\in (\bar{t}-\epsilon,\bar{t}+\epsilon)$.
We write that
\[
Q(x(t))^{-1}\Lambda(t) Q(x(t))^{-T}=\left[\begin{array}{cc}
\Lambda_1(t) & \Lambda_2(t) \\
\Lambda_2(t)^T & \Lambda_3(t)\\
\end{array}\right] .
\]
The equation $G(x(t))\Lambda(t)=0$ implies
\begin{eqnarray*}
0&= &Q(x(t))^TG(x(t))Q(x(t)) \cdot Q(x(t))^{-1}\Lambda(t) Q(x(t))^{-T} \\
& = &\left[\begin{array}{cc}
G_1(x(t))\Lambda_1(t) & G_1(x(t))\Lambda_2(t) \\
S(x(t))\Lambda_2(t)^T & S(x(t))\Lambda_3(t)\\
\end{array}\right].
\end{eqnarray*}
Thus, $\Lambda_1(t)=\Lambda_2(t)=0$ for all $t\in (\bar{t}-\epsilon,\bar{t}+\epsilon)$,
since  $G_1(x(t))$ is invertible. For convenience,
we write $G_i(x(t))$ as $G_i$ and $\Lambda_i(t)$ as $\Lambda_i$. Then, we have
\begin{eqnarray*}
\Lambda(t) &=& Q(x(t))\left[\begin{array}{cc}
0 & 0 \\
0 & \Lambda_3\\
\end{array}\right]Q(x(t))^{T}\\
&=& \left[\begin{array}{cc}
G_1^{-1}G_2 \Lambda_3 G_2^TG_1^{-T} & -G_1^{-1}G_2\Lambda_3  \\
-(G_1^{-1}G_2 \Lambda_3)^T & \Lambda_3  \\
\end{array}\right],
\end{eqnarray*}
and
\[
\begin{aligned}
\tr{  G \frac{d\Lambda(t)}{dt}}
& = \tr{ G_1\frac{d(G_1^{-1}G_2\Lambda_3G_2^TG_1^{-1})}{dt} }
-2 \tr{ G_2^T\frac{d(G_1^{-1}G_2\Lambda_3)}{dt} } +
\tr{ G_3\frac{d\Lambda_3}{dt} } \\
&=\tr{G_1\frac{d(G_1^{-1}G_2\Lambda_3)}{dt}G_2^TG_1^{-1} }+
\tr{ G_1G_1^{-1}G_2\Lambda_3\frac{d(G_2^TG_1^{-1})}{dt} }  \\
& \quad -2 \tr{ G_2^T\frac{d(G_1^{-1}G_2\Lambda_3)}{dt} } +
\tr{ G_3\frac{d\Lambda_3}{dt} }  \\
& = -\tr{ G_2^T\frac{d(G_1^{-1}G_2\Lambda_3)}{dt} } +
\tr{ G_2\Lambda_3\frac{d(G_2^TG_1^{-1})}{dt} } +
\tr{ G_3\frac{d\Lambda_3}{dt} } \\
&= -\tr{ G_2^T\frac{d(G_1^{-1}G_2)}{dt}\Lambda_3 }  -
\tr{ G_2^TG_1^{-1}G_2\frac{d\Lambda_3}{dt} }  \\
&\quad + \tr{ G_2\Lambda_3\frac{d(G_2^TG_1^{-1})}{dt} } +
\tr{ G_3\frac{d\Lambda_3}{dt} } .
\end{aligned}
\]
Note that $\Lambda_3$ is symmetric and
\[
\tr{ G_2\Lambda_3\frac{d(G_2^TG_1^{-1})}{dt} } =
\tr{  \frac{d( G_1^{-1} G_2^T)}{dt}  \Lambda_3 G_2^T} =
\tr{ G_2^T \frac{d( G_1^{-1} G_2^T)}{dt}  \Lambda_3 },
\]
so we can further get
\[
\begin{aligned}
\tr{  G(x(t)) \frac{d\Lambda(t)}{dt}}
&= \tr{ G_3\frac{d\Lambda_3}{dt} } -
\tr{ G_2^TG_1^{-1}G_2\frac{d\Lambda_3}{dt} }  \\
&= \tr{ S(x(t)) \frac{d\Lambda_3}{dt} }  =0,
\end{aligned}
\]
which follows from
$
S(x(t))= G_3 - G_2^TG_1^{-1}G_2 = 0.
$
Then, it holds that
\begin{eqnarray*}
\frac{\mathrm{d} f}{\mathrm{d}t}(x(\bar{t}))
& = & \frac{d\mathcal{L}(x(\bar{t}))}{dt} \\
& = & \nabla f(x(\bar{t}))^T\nabla x(\bar{t})-(\nabla G(x(\bar{t}))^* [\Lambda(t)])^T\nabla x(\bar{t})-
\tr{ G(x(\bar{t})) \frac{d\Lambda(\bar{t})}{dt} } \\
&=& (\nabla f(x(\bar{t}))-\nabla G(x(\bar{t}))^* [\Lambda(\bar{t})])^T\nabla x(\bar{t}) -
\tr{  G(x(\bar{t})) \frac{d\Lambda(\bar{t})}{dt} }  \\
&=& - \tr{ G(x(\bar{t}))  \frac{d\Lambda(\bar{t})}{dt} }   =  0.
\end{eqnarray*}
Since $\bar{t}$ is arbitrary in  $S_{\ell}$, we know that  $f(x(t))$ has  zero gradient
for all $t\in S_{\ell}$, and  by the piecewise smoothness of  $x(t)$, also for all $t\in T_{\ell}$.

Since $(t_i,t_{i+1})= \cup_{\ell=0}^mT_{\ell}$,
$f(x(t))$ has zero gradient in $(t_i,t_{i+1})$.
By integration, we have
\[
f( t_{i+1} ) - f( t_i ) =
\int_{t_i}^{ t_{i+1} } \frac{\mathrm{d} f(x(t))}{\mathrm{d}t}  dt  = 0.
\]
This holds for all intervals $(t_i,t_{i+1})$, so
\[
f(x^{(1)})=f(x(t_0))=f(x(t_1))=\cdots=f(x(t_N))=f(x^{(2)}).
\]
Since $(x^{(1)}, \Lambda^{(1)})$, $(x^{(2)}, \Lambda^{(2)})$
are arbitrary points in $W$,
we conclude that $f$ achieves a constant value on $W$.
\end{proof}
\medskip

\subsection{Step2 : decomposition of the critical set
 and SOS representations}

Under Assumption~\ref{ass:multi}, the set of complex critical points of \reff{nsdp} is
\begin{equation}\label{kktcri}
\mathcal{K} :=\left\{\begin{array}{l|l}
x \in \mathbb{C}^n &
\begin{array}{r}
\nabla f(x)-\nabla G(x)^*[\Theta(x)] = 0, \\
G(x)\Theta(x)=0
\end{array}
\end{array}\right\}.
\end{equation}
Suppose $\mathcal{K}$ has the decomposition of irreducible subvarieties:
\[
\mathcal{K}  \, = \,  W_1 \cup \cdots \cup W_s .
\]
By Lemma \ref{kktfini}, if $W_i \cap \re^n \ne \emptyset$,
then $f$ achieves a real constant value $v_i$ on  $W_i$.
Without loss of generality, assume that $f$ achieves $r$ distinct real values, ordered as
\[ v_1   <  v_2  < \cdots < v_r . \]
Let $\mathcal{K}_i$ be the union of all subvarieties $W_i$ on which $f$
achieves the real constant  value $v_i$ and let $\mathcal{K}_0$
be the union of all remaining $W_i$. Then,
\[
\mathcal{K}=\mathcal{K}_0 \cup \mathcal{K}_1 \cup \cdots \cup \mathcal{K}_r,\quad
\mathcal{K}_0 \cap \mathbb{R}^n=\emptyset,
\]
and $f \equiv v_i$ on $\mathcal{K}_i$ for $i=1, \ldots, r$.
Since the values $v_i$ are pairwise distinct, the subvarieties
$\mathcal{K}_0, \mathcal{K}_1, \ldots, \mathcal{K}_r$
are disjoint from each other.  Denote the ideal
\begin{equation}\label{equI}
J  \coloneqq  \ideal{ G \Theta,~\nabla f-\nabla G^*[\Theta] }.
\end{equation}
By the primary decomposition theorem \cite{ggm,Stu02},
there exist ideals $J_0,J_1,\dots,J_r \subseteq \mR[x]$ such that
\begin{equation}\label{decompI}
J  = J_0 \cap J_1 \cap \ldots \cap J_r
\end{equation}
and $\mathcal{K}_i=V_{\mathbb{C}}(J_i)$ for $i=0, \ldots, r$.
By \cite[Proposition 9]{sk09},  there exist polynomials
$g_1,\dots,g_{\ell_1}\in  \qmod{G,\Theta}$ such that
\begin{equation}\label{def:Gtheta}
\boxed{
\baray{c}
G(x) \succeq 0,~\Theta(x) \succeq 0 \quad \Longleftrightarrow \quad  \\
g_1(x)\geq 0, \dots,  g_{\ell_1}(x)\geq 0 .
\earay
}
\end{equation}

\begin{lem} \label{idealcons}
Let $J, J_0, \ldots, J_r$ be the  ideals defined above.
Suppose $v_{r_0}=f_{\min}$ for some $1\leq r_0\leq r$
and the set $J + \qmod{G,\Theta}$ is Archimedean. Then, we have:
\bit
\item[(i)] For $i=0$,  there exists $\sigma_0\in \Sigma[x]$ such that
$f-f_{\min}-\sigma_0\in J_0$.
\item[(ii)] For $i=1,\dots,r_0-1$, there exists $\sigma_i\in \mR[x]$ such that
\[
f-f_{\min}-\sigma_i\in J_i, \quad \sigma_i\in J_i+\qmod{G,\Theta} .
\]				
\item[(iii)] For $i=r_0$, there exists an integer $k_0>0$ such that for every $\epsilon>0$,
\[
f-f_{\min}+\epsilon-\sigma_{r_0,\epsilon} \in (J_{r_0})_{2k_0}
\quad \text{for some} \quad \sigma_{r_0,\epsilon} \in \Sigma[x]_{2k_0} .
\]
\item[(iv)] For $i=r_0+1,\dots,r$, there exists  $\sigma_i\in \Sigma[x]$ such that
$f-f_{\min}-\sigma_i\in J_i$.
\eit
\end{lem}
\begin{proof}
Up to shifting $f$ by a constant, we can generally assume $f_{\min}=0$.

(i) For $i=0$, $V_{\mathbb{R}}\left(J_0\right) = \emptyset $.
By Real Nullstellensatz \cite[Corollary 4.1.8]{bcr}, there exists  $\tau_0 \in \Sigma[x]$
such that $1+\tau_0 \in J_0$. Let
\[
\sigma_0:=\frac{1}{4}(f+1)^2+\frac{\tau_0}{4}(f-1)^2 \in \Sigma[x].
\]
Then, we have that
\[
f-\sigma_0=\frac{1}{4}(f+1)^2-\frac{1}{4}(f-1)^2-\sigma_0=-\frac{1+\tau_0}{4}(f-1)^2\in J_0.
\]

(ii)
For $i=1,\dots,r_0-1$,  we have that $v_i<f_{\min}=0$, and thus
\[
V_{\mathbb{R}}(J_i)\cap \{x\in \mR^n: G(x) \succeq 0,~\Theta(x) \succeq 0\}=\emptyset.
\]
It follows from \reff{def:Gtheta} that
\[
V_{\mathbb{R}}(J_i)\cap \{x\in \mR^n:~
g_1(x)  \geq 0,  \dots,  g_{\ell_1}(x)\geq 0\}  =  \emptyset.
\]
By Stengle's Positivstellensatz \cite[Corollary 4.4.3]{bcr}, there exists
\[
\phi  ~= \sum_{ \af = (\af_1, \ldots, \af_{\ell_1}) \in  \{0,1\}^{\ell_1}  }
\phi_{\alpha} g_1^{\alpha_1} \cdots g_{\ell_1}^{\alpha_{\ell_1}},
\quad \text{with} \quad  \phi_{\alpha}\in \Sigma[x],
\]
such that $2+\phi\in J_i$.
Note that $J \subseteq J_i$, and the set $J+\qmod{G,\Theta}$ is Archimedean.
The following statement holds:
\[
1+\phi(x)>0 \quad \text{for all} \quad
x\in V_{\mR}(J)\cap\{ G(x) \succeq 0,~\Theta(x) \succeq 0\}.
\]

By Putinar's Positivstellensatz \cite{putinar1993positive}, we have
$1+\phi\in J+\qmod{G,\Theta}\subseteq  J_i+\qmod{G,\Theta}$.  Let
\[
\sigma_i:=\frac{1}{4}(f+1)^2+\frac{1+\phi}{4}(f-1)^2 \in J_i+\qmod{G,\Theta}.
\]
Then, we have
\[
f-\sigma_i  = \frac{1}{4}(f+1)^2-\frac{1}{4}(f-1)^2-\sigma_i
=  -\frac{2+\phi}{4}(f-1)^2  \in  J_i.
\]

(iii)
For $i=r_0$,  we have $v_{i}=f_{\min}=0$, so
$f \equiv 0$ on $V_{\mathbb{C}}(J_{i})$.
By Hilbert's Strong Nullstellensatz \cite{blo}, there exists an integer
$\eta>0$ such that $f^{\eta} \in J_{r_0}$. For $\epsilon>0$, let
\[
s_{r_0,\epsilon}(x)  \coloneqq  \sqrt{\epsilon} \sum_{j=0}^{\eta-1}
\binom{\frac{1}{2}}{j} (f/\eps )^j,~~~
\sigma_{r_0,\epsilon}  \coloneqq   s_{r_0,\epsilon}^2.
\]
It can be verified directly that
\[
f+\epsilon-\sigma_{r_0,\epsilon} = \eps \sqrt{1 + f/\eps }^2 - \sigma_{r_0,\epsilon}
=\sum_{j=0}^{\eta-2} b_j(\epsilon) f^{\eta+j},
\]
where  $b_j(\epsilon)$ are real scalars depending only on $\epsilon$.
Note that $f^{\eta} \in J_{r_0}$ and the degree of $\sigma_{r_0,\epsilon}$
is independent of $\epsilon$. For sufficiently large integer $k_0$,
we have $\sigma_{r_0,\epsilon} \in \Sigma[x]_{2k_0}$ and
$f+\epsilon-\sigma_{r_0,\epsilon} \in (J_{r_0})_{2k_0}$ for all $\epsilon>0$.

(iv) For $i=r_0+1, \ldots, r$, we have that $v_i>f_{\min}=0$, i.e.,
$f /v_i-1 \equiv 0$ on $V_{\mathbb{C}}(J_i).$
By Hilbert's Strong Nullstellensatz \cite{blo}, there exists $\eta_i \in \mathbb{N}$
such that $ \left(f/v_i - 1 \right)^{\eta_i} \in J_i$. Let
\[
s_i   \coloneqq  \sqrt{v_i} \sum_{j=0}^{\eta_i-1}\binom{\frac{1}{2}}{j}
\left( f/v_i - 1 \right)^j,
\sigma_i:=s_i^2.
\]
Similarly to (iii), we have  $f-\sigma_i\in J_i$.
\end{proof}

\subsection{Step 3: proof of tightness}
To prove Theorem~\ref{thm:finite},  it suffices to show that there exists an order $k$ such that for all $\epsilon>0$, the following holds:
\begin{equation}\label{f-fmin=SOS}
f - f_{\min} +  \epsilon \in \ideal{G \Theta,\nabla f - \nabla G^*[\Theta] }_{2k}
+\qmod{G,\Theta}_{2k} .
\end{equation}

\bigskip \noindent
{\it Proof of Theorem \ref{thm:finite}.}
Up to  shifting $f$ by a constant, we  assume $f_{\min}=0$.
Suppose that $v_{r_0} = 0$ for some $1\leq r_0\leq r$.
Since the complex varieties $V_{\mC}(J_0), \dots, V_{\mC}(J_r)$ are disjoint from each other,
it follows from \cite[Lemma 3.3]{niejac} that there exist polynomials
$a_0, \ldots, a_r \in \mathbb{R}[x]$ such that
\[
a_0^2+\cdots+a_r^2-1 \in J,
\]
\[
a_i \in \bigcap_{i\neq j \in \{0,\dots,r\}} J_{j}
\quad \text{for} \quad i=0,\dots,r .
\]
By Lemma \ref{idealcons}, we have $f-\sigma_i \in J_i$ for  $i \neq r_0$ and $f+\epsilon-\sigma_{r_0,\epsilon} \in J_{r_0}$, then
$$
(f-\sigma_i) a_i^2 \in \bigcap_{i=0}^r J_i = J,~~~~\forall \, i \neq r_0,
$$
$$
(f+\epsilon-\sigma_{r_0,\epsilon}) a_{r_0}^2 \in  \bigcap_{i=0}^r J_i  = J,
$$
$$
(f+\epsilon)\left(1-a_{0}^2-\cdots-a_r^2\right) \in J.
$$
For $\epsilon>0$, let
\[
\sigma_{\epsilon}  \coloneqq  \sigma_{r_0,\epsilon}a_{r_0}^2+
\sum_{r_0 \neq i \in\{0, \ldots, r\}} \left(\sigma_i+\epsilon\right) a_i^2.
\]
Then, we have
\begin{eqnarray*}
f+\epsilon-\sigma_\epsilon & = & (f+\epsilon)\left(1-a_0^2-\cdots-a_r^2\right) \\
& & +\left(f+\epsilon-\sigma_{r_0,\epsilon}\right) a_{r_0}^2+
\sum_{r_0 \neq i \in\{0, \ldots, r\}}\left(f-\sigma_i\right) a_i^2 \in J.
\end{eqnarray*}
By Lemma \ref{idealcons} (iii),
there exists $k_0\in \mathbb{N}$ such that for all  $\epsilon>0$, we have
\[
\sigma_{r_0,\epsilon},~f+\epsilon-\sigma_{r_0,\epsilon}
\in J_{2k_0}+\qmod{G,\Theta}_{2k_0}.
\]
Hence,      when $k$ is sufficiently large, it holds that
\[
\sigma_{\epsilon},\,\,f+\epsilon-\sigma_{\epsilon} \in J_{2k}+\qmod{G,\Theta}_{2k}.
\]
This implies that \reff{f-fmin=SOS} holds and
$f_{k,\mathrm{sos}} \geq f_{\min}-\epsilon$ for all $\epsilon>0$.
Since $f_{k,\mathrm{sos}} \leq f_{\min}$ for all $k$, we  conclude $f_{k,\mathrm{sos}}=f_{\min}$
for all $k$ sufficiently large.
\hfill\tallopenbox

It is possible that the optimal value of the SOS relaxation \reff{sos} is not attained for all the big relaxation orders in practice, even when finite convergence holds. Consider, for example, the unconstrained optimization problem
\begin{equation} \label{sp1}
	\min \quad  x_1^8+x_2^8+x_3^8+M(x),  	
\end{equation}
where	$M(x)=x_1^2 x_2^4+x_2^2 x_1^4+x_3^6-3 x_1^2 x_2^2 x_3^2$ is the Motzkin polynomial. Then, the equation in \reff{sec3:nsdp:eq}  reads as $\nabla f(x)=0$, equivalently,
\[
8x_i^7+\frac{\partial M}{\partial x_i}(x)=0,\quad i=1,2,3.
\]
The set $\{x\in \mR^n: \nabla f(x)=0\}$ is compact,
so the ideal $\ideal{\nabla f}$ is Archimedean.
Thus the assumptions of Theorem 4.2 are satisfied, and the hierarchy \reff{mom}-\reff{sos} is tight.
However, the optimal value of \reff{sos} is not attainable for all sufficiently large  $k$,
because  it is not a sum of squares modulo $\ideal{\nabla f}$ as shown in \cite{NDS06}.

Recall that $g_1,\dots,g_{\ell_1}$ are polynomials as in \reff{def:Gtheta}.
For convenience, let $h_1,\cdots, h_{\ell_2}$ be polynomials such that
\begin{equation}\label{ideal:c1cell}
\ideal{G \Theta,  \nabla f - \nabla G^*[\Theta] }  ~=~
\ideal{ h_1,\dots, h_{\ell_2} }.
\end{equation}
The following result is a direct consequence of \reff{f-fmin=SOS}, as shown in the proof of Theorem~\ref{thm:finite}.

\begin{cor}  \label{thocor}
Under the assumptions of Theorem~\ref{thm:finite},
there exists an order $k > 0$ such that for all $\epsilon>0$, it holds
\be
\left\{ \baray{c}
f-f_{\min}+\epsilon=\sum\limits_{i=1}^{\ell_2} \phi_i^{\epsilon} h_i
+ \psi_0^{\epsilon}  + \sum\limits_{j=1}^{\ell_1} \psi_j^{\epsilon} g_j, \\
\phi_1^{\epsilon},\dots,\phi_{\ell_2}^{\epsilon}\in \mR[x],
\psi_0^{\epsilon},\dots,\psi_{\ell_1}^{\epsilon}\in \Sigma[x], \\[1.3ex]
\deg(\phi_i^{\epsilon}h_i)\leq 2k, ~ \deg( \psi_0^{\epsilon} ) \le 2k,
~\deg(\psi_j^{\epsilon} g_j) \leq 2k .
\earay \right.
\end{equation}
\end{cor}

\section{Detection of tightness and extraction of minimizers}
\label{sc:detect}
In computational practice, flat truncation
is typically used to detect tightness of Moment--SOS relaxations and
to extract minimizers (see \cite{CF05,HenLas05,hdlb,Lau05,nie2013certifying}).
Let $y^{(k)}$ be a minimizer of the moment relaxation~\reff{mom} for the order $k$.
The flat truncation is said to hold for $y^{(k)}$
if there exists an integer $t \in[ d_0, k]$ such that
\begin{equation}\label{FT:y*}
\operatorname{rank} M_{t} [ y^{(k)} ] \, = \,
\operatorname{rank} M_{t-d_0}[ y^{(k)} ] .
\end{equation}
When the relation \reff{FT:y*} holds, the truncation $y^{(k)}|_{2t}$ admits an $r$-atomic probability measure
$\mu$ with $r=\operatorname{rank} M_{t} [ y^{(k)} ]$ and the support of $\mu$ is contained in $K$.
That is, there exist $r$ distinct points $u_1,\ldots, u_r \in K$
and positive scalars $\tau_j$ such that
\[
y^{(k)}|_{2t}   =  \sum_{j=1}^r \tau_j [u_j]_{2t}, \quad
\sum_{j=1}^r \tau_j=1.
\]
Furthermore, the moment relaxation  \reff{mom} is tight  (i.e., $f_{k,\mathrm{mom}}=f_{\min}$) and $u_1, \dots, u_r$ are minimizers of \reff{nsdp:equ}.
The points $u_j$ can be computed by
eigenvalue computations and Schur decomposition \cite{HenLas05}.
Under some general assumptions, we show that every minimizer
$y^{(k)}$ of  \reff{mom} satisfies \reff{FT:y*}
when the relaxation order $k$ is sufficiently large.

\begin{theorem}   \label{thm:flat}
Suppose the optimal value $f_{\min}$ of \reff{nsdp} is achievable at a KKT point and
Assumption~\ref{ass:multi} holds.  Let $J$ be the ideal as in \reff{equI},
and let $g_1,\dots,g_{\ell_1}$ be the polynomials as in \reff{def:Gtheta}.
Assume at least one of the following conditions holds:
\bit	
\item[(i)] The real variety $V_{\mR}(J)$ is finite;

\item[(ii)] There exists $ \psi \in \qmod{G,\Theta}$ such that $\psi(x)<0$
for all $x\in V_{\mR}(J)$ with $f(x)<f_{\min}$, and the problem
\begin{equation} \label{nsdp:loc}
\left\{ \baray{cl}
\min & f(x)  \\
\st &  \nabla	f(x)-\nabla G(x)^*[\Theta(x)] = 0, \\
&  G(x)\Theta(x)=0,  \, \psi(x)\geq 0,
\earay \right.
\end{equation}
has finitely many minimizers.
\eit
Then, every minimizer $y^{(k)}$ of the moment relaxation \reff{mom}
satisfies the flat truncation
when the relaxation order $k$ is sufficiently large.
\end{theorem}
\begin{proof}
Up to shifting $f$ by a constant, we  assume $f_{\min}=0$ for convenience.

(i) The conclusion follows from \cite[Proposition 4.6]{LLR08}.

(ii) Under the given assumptions,  problems \reff{nsdp:equ} and \reff{nsdp:loc} have the same optimal value, which is equal to $f_{\min}$.
The $k$th order SOS relaxation of \reff{nsdp:loc} is
\begin{equation}\label{sos:loc}
\left\{ \baray{rl}
\min & \gamma  \\
\st &  	f-\gamma \in \ideal{G  \Theta, \nabla f-
\nabla G^*[\Theta]}_{2k}+\qmod{\psi}_{2k}.\\
\earay \right.
\end{equation}
The dual problem of \reff{sos:loc} is the $k$th order moment relaxation:
\begin{equation}\label{moment:loc}
\left\{ \baray{rl}
\min & \langle f, z\rangle  \\
\st &  L^{(k)}_{G\cdot \Theta}(z)=0,~
L^{(k)}_{\nabla f-\nabla G^*[\Theta] }(z)=0,\\
& M_k(z) \succeq 0,~ L_{\psi}^{(k)}(z) \succeq 0,\\
&z_0 =1, \, z \in \re^{ \N_{2k}^n }.\\
\earay \right.
\end{equation}
Denote by $f^{\prime}_{k,\mathrm{sos}}$, $f^{\prime}_{k,\mathrm{mom}}$
the optimal values of \reff{sos:loc}
and \reff{moment:loc}, respectively.

Let $J_0,J_1,\dots,J_r \subseteq \mR[x]$ be the ideals  defined in \reff{decompI},
with $v_{r_0}=f_{\min}$ for some $1\leq r_0 \leq r$.
Note that $f \equiv v_i$ on $V_{\mathbb{C}}(J_i)$ for $i=1, \ldots,r$,
and $V_{\mathbb{C}}(J_0)\cap \mR^n=\emptyset$.  Let $\varphi_{r_0}, \dots, \varphi_r$
be the interpolating  polynomials such that for
$\varphi_i(v_i)=1$ and $\varphi_i(v_j)=0$ for $ j \neq i$. Define
\[
\sigma  \, \coloneqq  \,  \sum\limits_{i=r_0}^{r}v_i\cdot \left(\varphi_i(f)\right)^2.
\]
Then, the polynomial
$\hat{f}:=f-\sigma$ is identically zero  on the set
\[
V_{\mR}(J)\cap \left\{x \in \mathbb{R}^n:   \psi(x) \geq 0\right\}.
\]
By the Positivstellensatz \cite[Corollary 4.4.3]{bcr},
there exist an integer $\ell > 0$ and $\sigma^*\in \qmod{\psi}$
such that $\hat{f}^{2 \ell}+\sigma^*\in J$. For  $\epsilon>0$, let
\[
\begin{gathered}
\phi_\epsilon=-\frac{1}{2 \ell} \epsilon^{1-2\ell}\Big(\hat{f}^{2 \ell}+\sigma^*\Big),  \\
\theta_\epsilon  = \epsilon \Big(1+\hat{f} / \epsilon  +
\frac{1}{2 \ell}(\hat{f} / \epsilon)^{2 \ell}\Big)  +
\frac{1}{2 \ell} \epsilon^{1-2\ell} \sigma^*.
\end{gathered}
\]
Since the polynomial $1+t+\frac{1}{2\ell}t^{2\ell}$ is an SOS,
there exists $k_1>0$ such  that for all $\epsilon>0$,
we have that $\phi_\epsilon\in J_{2k_1}$ and $\theta_\epsilon\in \qmod{G,\psi}_{2k_1}$.
Note that
\[
f+\epsilon  = \hat{f}+\epsilon+\sigma  =  \phi_\epsilon+\sigma_\epsilon+\sigma.
\]
Then, there exists an order $k_2\geq k_1$ such that for all $\epsilon>0$,
we have $f+\epsilon\in J_{2k_2}+\qmod{\psi}_{2k_2}$.
Thus, for all $\epsilon>0$,  $-\epsilon$ is  feasible for \reff{sos:loc}, which implies $f^{\prime}_{k,\mathrm{sos}}=f^{\prime}_{k,\mathrm{mom}}=f_{\min}$ for all $k\geq k_2$.
Since  $ \psi \in \qmod{G,\Theta}$, we also have
$f_{k,\mathrm{sos}}=f_{k,\mathrm{mom}}=f_{\min}$ for  $k$ sufficiently large.
Since  every minimizer $y^{(k)}$ of \reff{mom} also gives a minimizer of \reff{moment:loc},
by using subvectors, the conclusion follows from \cite[Theorem 2.6]{nie2013certifying}.
\end{proof}

\section{Numerical experiments}
\label{sec:num}
This section  presents numerical examples of the strengthened hierarchy
\reff{mom}--\reff{sos} for solving MPO \reff{nsdp}.
The computation is implemented in MATLAB R2024a, on a Lenovo Laptop with CPU@1.40 GHz and RAM 32 {GB}.
The relaxations \reff{mom}--\reff{sos} are modeled using the software Yalmip \cite{yalmip},
and the semidefinite programs are solved by SeDuMi \cite{sturm}.

For a degree $k$, the $k$th order standard moment relaxation of \reff{nsdp} is
\begin{equation}
\label{p_mom}
\left\{
\begin{array}{rl}
\min & \langle f,y\rangle \\
\st & L^{(k)}_G(y)\succeq 0,\,M_k(y)\succeq 0,\\
& y_0=1,\,y\in \r^{\n_{2k}^n}.
\end{array}\right.
\end{equation}
The dual  problem of \reff{p_mom} is the $k$th order SOS relaxation:
\begin{equation}
\label{p_sos}
\left\{
\begin{array}{rl}
\max &\gamma  \\
\st & f-\gamma \in \qmod{G}_{2k}.
\end{array}
\right.
\end{equation}
We compare the performance of the strengthened relaxations \reff{mom}--\reff{sos}
with the standard  relaxations  \reff{p_mom}--\reff{p_sos}.
The obtained lower bounds and consumed computational times are presented in the tables.
The columns labeled ``with LME" correspond to  the strengthened relaxations  \reff{mom}--\reff{sos}  using the  polynomial Lagrange multiplier expression (LME) $\Lambda=\Theta(x)$,
while the columns labeled ``without LME" represent the standard relaxations \reff{p_mom}--\reff{p_sos}.
All computational times are reported in seconds, and numerical results are rounded to four decimal digits for clarity.

\begin{example}\label{exm7.1}
Consider the matrix polynomial optimization:
\begin{equation} \nonumber
\left\{ \baray{rl}
\min & x_1^2+x_2^2\\
\st &  	\begin{bmatrix}
x_1^2-2 & \frac{1}{2}x_1x_2 \\
\frac{1}{2}x_1x_2 & x_2^2-2 \\
\end{bmatrix}\succeq 0.
\earay \right.
\end{equation}
The optimal value $f_{\min}=8$ and the minimizers are $(\pm 2, \pm 2)$. By Example \ref{exmLME} (ii), we have the Lagrange multiplier matrix expression
\[
\Theta(x)=\begin{bmatrix}
\frac{1}{2}x_1^2 & x_1x_2-\frac{3}{16}x_1^3x_2-\frac{3}{16}x_1x_2^3\\
x_1x_2-\frac{3}{16}x_1^3x_2-\frac{3}{16}x_1x_2^3 & \frac{1}{2}x_2^2
\end{bmatrix}.
\]
The computational results  are presented  in Table \ref{ex5.5}.
For the strengthened hierarchy, we get $f_{3,\mathrm{sos}} = f_{\min}$
and all minimizers are extracted at the order $k=3$.
In contrast,  the standard hierarchy  converges at the order $k=8$, which is much slower.

\begin{table}[H]
\centering
\caption{Computational results for Example \ref{exm7.1}}
\label{ex5.5}
\renewcommand{\arraystretch}{1.1}
\begin{tabular}{|c|c|c|c|c|}
\hline
\multirow{2}{*}{order $k$} & \multicolumn{2}{c|}{without LME} & \multicolumn{2}{c|}{with LME} \\
\cline{2-5}
& lower bound & time & lower bound & time \\
\hline
3 & 4.0003 & 0.2358 & 8.0000 & 0.2921 \\
\hline
4 & 4.0260 & 0.3463 & \multicolumn{1}{c}{} & \multicolumn{1}{c|}{} \\
\cline{1-3}
5 & 4.2806 & 0.5764 & \multicolumn{1}{c}{} & \multicolumn{1}{c|}{} \\
\cline{1-3}
6 & 5.3344 & 0.9873 & \multicolumn{1}{c}{} & \multicolumn{1}{c|}{} \\
\cline{1-3}
7 & 7.3760 & 1.7722 & \multicolumn{1}{c}{} & \multicolumn{1}{c|}{} \\
\cline{1-3}
8 & 8.0000 & 2.0490 & \multicolumn{1}{c}{} & \multicolumn{1}{c|}{} \\
\cline{1-5}
\end{tabular}
\end{table}
\end{example}

\begin{example}\label{exm7.5}
Consider the matrix polynomial optimization:
\begin{equation} \nonumber
\left\{ \baray{rl}
\min & (x_1-1)(x_4-1)-(x_2-1)^2+0.1\cdot(x_1+x_4+x_6-3) \\
\st & \left[\begin{array}{ccc}
x_1& x_2&x_3 \\
x_2 &x_4&x_5\\
x_3&x_5&x_6\\
\end{array}\right]\succeq 0.
\earay \right.
\end{equation}
The optimal value $f_{\min}=-0.3$ and the unique minimizer is $0$. We have
$
P_1 = \operatorname{diag}(1,2,1,2,2,1).
$
As shown in Example \ref{exmLME} (i),
the equation $L(x)P(x)=I_6$ is satisfied for
\[
L(x) \, =  \,  \begin{bmatrix}
\operatorname{diag}(1,\frac{1}{2},1,\frac{1}{2},\frac{1}{2},1) & \,\,0
\end{bmatrix}.
\]
We have the Lagrange multiplier matrix expression
\[
\Theta(x)  =
\left[\begin{array}{ccc}
f_{x_1} & \frac{1}{2} f_{x_2} &   \frac{1}{2} f_{x_3} \\
\frac{1}{2} f_{x_2} &  f_{x_4} & \frac{1}{2} f_{x_5}  \\
\frac{1}{2} f_{x_3} &  \frac{1}{2} f_{x_5}  &  f_{x_6} \\
\end{array}\right].
\]
The
computational results are shown in Table~\ref{ex5.1}.
For the strengthened hierarchy, we get
$f_{1,\mathrm{sos}} = f_{\min}$ and the minimizer  is extracted at the order $k=2$.
In contrast, the standard hierarchy only yields a looser lower bound  -23.4823  at  the order $k=5$, and it takes approximately $2401$ seconds.
\begin{table}[htb]
\centering
\caption{Computational results for Example \ref{exm7.5}}
\label{ex5.1}
\renewcommand{\arraystretch}{1.1}
\begin{tabular}{|c|c|c|c|c|}
\hline
\multirow{2}{*}{order $k$} & \multicolumn{2}{c|}{without LME} & \multicolumn{2}{c|}{with LME} \\
\cline{2-5}
& lower bound & time & lower bound & time \\
\hline
1 & $-\infty$ & 0.1347 & -0.3000 & 0.1504 \\
\hline
2 & -1.2676$\cdot 10^{5}$ & 0.4608 & -0.3000 & 0.3588 \\
\hline
3 & -128.7776 & 6.6117 & \multicolumn{1}{c}{} & \multicolumn{1}{c|}{} \\
\cline{1-3}
4 & -42.6204 & 112.2368 & \multicolumn{1}{c}{} & \multicolumn{1}{c|}{} \\
\cline{1-3}
5 & -23.4823 & 2401.2189 & \multicolumn{1}{c}{} & \multicolumn{1}{c|}{} \\
\cline{1-5}
\end{tabular}
\end{table}
\end{example}

\begin{example}\label{exm7.6}
Consider the matrix polynomial optimization:
\begin{equation}  \nonumber
\left\{ \baray{rl}
\min & (1+2x_1)(1+2x_3)-(x_3-x_4)^2+x_1+x_3+x_6 \\
\st &  \left[\begin{array}{ccc}
1+2x_1& x_3-x_4&x_5-x_6 \\
x_3-x_4 &1+2x_3&x_1-x_2 \\
x_5-x_6&x_1-x_2&1+2x_6  \\
\end{array}\right]\succeq 0 .
\earay \right.
\end{equation}
The minimum value $f_{\min}=-\frac{3}{2}$
and the unique minimizer is $-\frac{1}{2}\cdot(1,1,1,1,1,1)$.
As shown in Example \ref{exmLME} (i), the equation $L(x)P(x)=I_6$ is satisfied for
\[
L(x) \, =  \,  \begin{bmatrix}
    \frac{1}{2} & \frac{1}{2} & 0 & 0 & 0 & 0&0&\cdots & 0\\
    0 & 0&0&-\frac{1}{2} & 0&0&0&\cdots & 0\\
    0&0&\frac{1}{2}&\frac{1}{2}&0&0&\cdots & 0\\
    0 & 0& 0&0 & \frac{1}{2} & 0&0&\cdots & 0\\
    0 & -\frac{1}{2} & 0 & 0&0&0&0&\cdots & 0\\
    0&0&0&0&\frac{1}{2}&\frac{1}{2}&0&\cdots & 0
\end{bmatrix}.
\]
We get the Lagrange multiplier matrix expression
\[
\Theta(x)  = \frac{1}{2}
\left[\begin{array}{ccc}
f_{x_1} + f_{x_2} &  -f_{x_4} &    f_{x_3}+f_{x_4} \\
-f_{x_4} &  f_{x_5} &   -f_{x_2}  \\
f_{x_3}+f_{x_4} &  -f_{x_2}  &  f_{x_5}+f_{x_6} \\
\end{array}\right].
\]
The computational results are presented in Table~\ref{ex5.2}.
For the strengthened hierarchy, we get $f_{1,\mathrm{sos}} = f_{\min}$
and the minimizer is extracted at the order  $k=2$.
In contrast, the standard hierarchy  converges much slower.
\begin{table}[htb]
\centering
\caption{Computational results for Example \ref{exm7.6}}
\label{ex5.2}
\renewcommand{\arraystretch}{1.1}
\begin{tabular}{|c|c|c|c|c|}
\hline
\multirow{2}{*}{order $k$} & \multicolumn{2}{c|}{without LME} &
\multicolumn{2}{c|}{with LME} \\
\cline{2-5}
& lower bound & time & lower bound & time \\
\hline
1 & $-\infty$ & 0.1289 & -1.5000 & 0.1792 \\
\hline
2 & -6.9594$\cdot 10^{5}$ & 0.6131 & -1.5000 & 0.4168 \\
\hline
3 & -26.6785 & 7.4711 & \multicolumn{1}{c}{} & \multicolumn{1}{c|}{} \\
\cline{1-3}
4 & -1.5000 & 32.8910 & \multicolumn{1}{c}{} & \multicolumn{1}{c|}{} \\
\cline{1 -5}
\end{tabular}
\end{table}
\end{example}  		
\begin{example}\label{exm7.2}
Consider the matrix polynomial optimization:
\begin{equation} \nonumber
\left\{ \baray{rl}
\min & x_1^3+x_2^3+x_3^3+4x_1x_2x_3-x_1(x_2^2+x_3^2)-x_2(x_1^2
+x_3^2)-x_3(x_2^2+x_1^2) \\
\st &  	\left[\begin{array}{ll}
x_1 & x_1x_2-1\\
x_1x_2-1& x_2x_3-1\\
\end{array}\right]\succeq 0.
\earay \right.
\end{equation}
The objective is a variation of Robinson's form \cite{reznick2000some}.
We have the Lagrange multiplier matrix expression $\Theta(x)$ as in Example~\ref{exmLME} (iii).
The computational results are presented in Table~\ref{ex5.4}.
For the strengthened hierarchy, we get $f_{3,\mathrm{sos}} = f_{\min}$
and the minimizer $(0.6798,1.0876,1.0113)$ is extracted at the order $k=5$.
In contrast, the standard hierarchy converges at the order $k=8$, which is much slower.
\begin{table}[H]
\centering
\caption{Computational results for Example \ref{exm7.2}}
\label{ex5.4}
\renewcommand{\arraystretch}{1.1}
\begin{tabular}{|c|c|c|c|c|}
\hline
\multirow{2}{*}{order $k$} & \multicolumn{2}{c|}{without LME} & \multicolumn{2}{c|}{with LME} \\
\cline{2-5}
& lower bound & time & lower bound & time \\
\hline
2 & $-\infty$ & 0.1333 & $-\infty$ & 0.1743 \\
\hline
3 & -1.0669$\cdot 10^{7}$ & 0.2902 & 0.8479 & 0.2808 \\
\hline
4 & -6.6884$\cdot 10^{4}$ & 0.7089 & 0.8479 & 0.6819 \\
\hline
5 & -3.6965$\cdot 10^{3}$ & 1.8897 & 0.8479 & 1.5210 \\
\hline
6 & -563.0493 & 5.2481 & \multicolumn{1}{c}{} & \multicolumn{1}{c|}{} \\
\cline{1-3}
7 & -142.4861 & 13.5362 & \multicolumn{1}{c}{} & \multicolumn{1}{c|}{} \\
\cline{1-3}
8 & 0.8479 & 19.6211 & \multicolumn{1}{c}{} & \multicolumn{1}{c|}{} \\
\cline{1-5}
\end{tabular}
\end{table}
\end{example}

\begin{example}
\label{exlower}
    {Consider the matrix polynomial optimization:
    \be  \nonumber
\left\{ \baray{rl}
\min & x_1x_2-x_2x_3 \\
\st & \left[\begin{array}{cc}
	x_1-x_2& x_2+x_3^2 \\
	x_2+x_3^2 & x_3\\
\end{array}\right]\succeq 0.
\earay \right.
\ee
The Lagrange multiplier matrix expression is
\[
\Theta(x)=\begin{bmatrix}
    f_{x_1} & \frac{1}{2}f_{x_1}+\frac{1}{2}f_{x_2}\\
    \frac{1}{2}f_{x_1}+\frac{1}{2}f_{x_2} & f_{x_3}-2x_3f_{x_1}-2x_3f_{x_2}
\end{bmatrix}.
\]
    The computational results are shown in Table \ref{table_lower}. For the strengthened hierarchy, we get $f_{3,\mathrm{sos}}=f_{\min}$ and the minimizer $(0.3375,
    0.0829,
    0.5348)$ is  extracted at the order $k=4$. In contrast, the standard hierarchy  converges much slower.
\begin{table}[htb]
	\centering
	\caption{Computational results for Example \ref{exlower}}
	\label{table_lower}
	\begin{tabular}{|c|c|c|c|c|}
		\hline
		\multirow{2}{*}{order $k$}& \multicolumn{2}{c|}{without LME}& \multicolumn{2}{c|}{with LME}\\
		\cline{2-5}  
		& lower bound &time &lower bound &time\\
		\hline
		3 &-6.2083$\cdot10^3 $&1.2737& -0.0164&1.2208\\
		\hline
		4 & -298.9214 &2.5482&-0.0164&2.3834\\
		\hline
		5 &-32.9141&6.1536& & \\
		\hline
        6 & -3.7074 & 13.3707 & & \\ \hline
        7 & -2.1665 & 29.6775 & & \\ \hline
        8 & -1.5234 & 65.3710 & & \\ \hline
	\end{tabular}
\end{table}}

\end{example}

\section{Conclusions and discussions}
\label{sc:dis}
We propose a strengthened Moment-SOS hierarchy for solving matrix polynomial optimization, by using Lagrange multiplier matrix expressions.
First, we show that polynomial  expressions for Lagrange multiplier matrices exist under the  nondegeneracy condition.
Second, we prove that this strengthened hierarchy is tight under some general assumptions.
Third, we show that every minimizer of the moment relaxation  has a flat truncation
when the relaxation order is sufficiently large.
The strong performance of this strengthened Moment-SOS hierarchy
is demonstrated through various numerical experiments.

We remark that the MPO \reff{nsdp} can be reformulated as the semi-infinite program (SIP)
\begin{equation} \label{sp10}
\left\{ \baray{rl}
\min & f(x)  \\
\st &  	y^TG(x)y  \geq 0 \quad   \forall \, y \in \mR^m:  \|y\|^2 \le 1. \\
\earay \right.
\end{equation}
However,  \reff{sp10} is more challenging to solve, as it involves a universal quantifier.
There exist asymptotically convergent Moment–SOS relaxation methods for solving SIPs \cite{HuKlepNie25,HuNie25}, but none of these methods
possesses finite convergence.
In fact, their convergence speed may be very slow.
%
%
To illustrate this, we report the performance of the SIP method in \cite{HuNie25} for solving Examples \ref{exm7.1}–\ref{exlower}:
\bit
\item For Example \ref{exm7.1}, the SIP method  
produces the lower bound $2.0000$ after $100$ iterations, 
which took around $48.03$ seconds.
\item 
For Example \ref{exm7.5}, the SIP method  
produces a lower bound of $-120.5554$ after $100$ iterations, 
which took around $77.09$ seconds.
\item 
For Example \ref{exm7.6}, the SIP method produces a lower bound of $-297.4388$ after $100$ iterations,
which took around $67.62$ seconds.
\item 
For Example \ref{exm7.2}, the SIP method produces a lower bound of $-44.2973$ after $100$ iterations, 
which took around $41.04$ seconds.
\item
For Example \ref{exlower}, the SIP method produces a lower bound of $-7.0711$ after $100$ iterations, 
which took around $17.82$ seconds. 
\eit
In contrast, the strengthened Moment-SOS method with LMEs
converges much faster, so it is 
more suitable for solving  \reff{nsdp}.

\bigskip

\noindent
{\bf Funding}
Jiawang Nie and Jiajia Wang are partially supported by
the NSF grants DMS-2110780 and  DMS-2513254.

\end{document}